\documentclass[reqno,12pt]{amsart}
\usepackage{amssymb}
\usepackage{mathrsfs}
\usepackage{amsmath,amssymb}
\usepackage{paralist}
\usepackage{graphics} %% add this and next lines if pictures should be in esp format
\usepackage{epsfig} %For pictures: screened artwork should be set up with an 85 or 100 line screen
\usepackage[colorlinks = true]{hyperref}
\hypersetup{urlcolor = red, citecolor = blue}
\usepackage[top = 1in,bottom = 1in,left = 1in,right = 1in]{geometry}
\usepackage{cite}
\usepackage{lineno}
\newtheorem{thm}{Theorem}[section]

\newtheorem{lem}[thm]{Lemma}

\theoremstyle{definition}
\newtheorem{defn}{Definition}[section]

\numberwithin{equation}{section}
\allowdisplaybreaks[3]

\begin{document}

% \linenumbers  %添加行号
% \let\oldalign\align
% \let\oldendalign\endalign
% \renewenvironment{align}{\linenomathNonumbers\oldalign}{\oldendalign\endlinenomath}
% \pagewiselinenumbers  %每页重新开始添加行号

%\title[chemotaxis system with singular density-suppressed motility and superlinear consumption]
%{Global solutions to a chemotaxis system with singular density-suppressed motility and superlinear consumption}
\title[ZHANG AND LI] 
{Boundedness in a two-dimensional doubly degenerate nutrient taxis system}
\author[Zhang]{Zhang Zhiguang}%
\address{School of Mathematics, Southeast University, Nanjing 211189, P. R. China}
\address{School of Mathematics and Statistics, Chongqing Three Gorges University, Wanzhou 404020 , P. R. China}
\email{guangz$\_$z@163.com}

\author[Li]{Li Yuxiang}
\address{School of Mathematics, Southeast University, Nanjing 211189, P. R. China}
\email{lieyx@seu.edu.cn}

\thanks{Supported in part by National Natural Science Foundation of China (No. 12271092, No. 11671079) and the  Science and Technology Research Program of Chongqing Municipal Education Commission (No. KJQN202201226).}



% \date{}%
% \dedicatory{}%
% \commby{}%
% ----------------------------------------------------------------
\begin{abstract}
In this work, we study the no-flux initial-boundary value problem for the doubly degenerate nutrient taxis system 
\begin{align}
\begin{cases}\tag{$\star$}\label{eq 0.1} 
u_t=\nabla \cdot(u v \nabla u)-\chi \nabla \cdot\left(u^{2} v \nabla v\right)+\ell u v, & x \in \Omega, t>0, \\ v_t=\Delta v-u v, & x \in \Omega, t>0  
\end{cases}
\end{align} 
in a smoothly bounded convex domain $\Omega \subset \mathbb{R}^2$, where $\chi>0$ and $\ell \geq 0$. In this paper, we present that for all reasonably regular initial data, the model \eqref{eq 0.1} possesses a global bounded weak solution which is continuous in its first and essentially smooth in its second component.
\end{abstract}
\maketitle

\section{Introduction and main results}\label{section1}

Patterns can be widely observed in nature; for example, bacterial colonies can manifest a broad spectrum of morphological aggregation patterns in response to adverse environmental conditions. Notably, species such as Bacillus subtilis are recognized for their capacity to exhibit diverse patterns, which depend on the softness of the medium, nutrient concentration, and temperature. The experimental results documented by Fujikawa \cite{1992-PASMaiA-Fujikawa}, Fujikawa and Matsushita \cite{1989-JotPSoJ-FujikawaMatsushita}, and Matsushita and Fujikawa \cite{1990-PASMaiA-MatsushitaFujikawa} provided compelling evidence of the nuanced collective dynamics exhibited by Bacillus subtilis populations under conditions of nutrient scarcity. Particularly noteworthy is the propensity for these populations to undergo the formation of intricate patterns, including the emergence of snowflake-like
population distributions. Such phenomena are not sporadic occurrences but rather recurrent features observed within nutrient-poor environments, suggesting a generic trait inherent to the system dynamics. 

Following experimental evidence indicating limitations in bacterial motility near regions of low nutrient concentration, as a mathematical description for such processes, Kawasaki et al., \cite{1997-JoTB-Kawa} proposed the doubly degenerate parabolic system
\begin{equation}\label{-A2}
\begin{cases}u_t=\nabla \cdot (D_u(u,v) \nabla u)+ \ell u v, & x \in \Omega, t>0, \\ v_t=D_v \Delta v-  u v, & x \in \Omega, t>0,\end{cases}
\end{equation}
where $D_u(u,v)$ is the diffusion coefficient of the bacterial cells depending on $u$ and $v$, and $D_v$ is the constant diffusion coefficient of the nutrients, $\ell \geq 0$. They proposed the simplest one as $D_u(u,v)=uv$. To describe the formation of the said aggregation pattern with more accuracy, Leyva, M\'{a}laga, and Plaza \cite{2013-PA-LeyvaMalagaPlaza} extended the above degenerate diffusion model to the following doubly degenerate nutrient taxis system 
\begin{align}\label{-A1-1}
\begin{cases}u_t=\nabla \cdot(u v \nabla u)-\chi \nabla \cdot\left(u^{2}v \nabla v\right)+ \ell u v, & x \in \Omega, t>0, \\ v_t=\Delta v- u v, & x \in \Omega, t>0,  \end{cases}
\end{align}
where $\chi>0$ and $\ell \geq 0$. In the two-dimensional setting, numerical simulations in \cite{1997-JoTB-Kawa, 2000-AiP-Ben-JacobCohenLevine, 2013-PA-LeyvaMalagaPlaza} indicated that, depending on the initial data and parameter conditions, the model \eqref{-A1-1} could generate rich branching pattern, which is very close to that observed in biological experiments.

Recently, Winkler \cite{2022-CVPDE-Winklera} first considered a special from of \eqref{-A1-1}, namely,
\begin{align}\label{-A1-2}
\begin{cases}u_t=\nabla \cdot(u v \nabla u)+ \ell u v, & x \in \Omega, t>0, \\ v_t=\Delta v- u v, & x \in \Omega, t>0,  \end{cases}
\end{align}
in a smoothly bounded convex domain $\Omega \subset \mathbb{R}^n$, where $n \geq 1$ and $\ell \geq 0$, and proved the global existence of weak solutions to the associated initial-boundary value problem under some assumptions on initial data. In \cite{2021-TAMS-Winkler}, Winkler considered one-dimensional doubly degenerate cross-diffusion system 
\begin{align}\label{-A1-3}
\begin{cases}u_t=\left(u v u_x\right)_x-\left(u^2 v v_x\right)_x+u v, \\
v_t=v_{x x}-u v, \end{cases}
\end{align}
where the initial data in \eqref{-A1-3} is assumed to be such that 
\begin{align*}
\begin{cases}u_0 \in C^{\vartheta}(\bar{\Omega}) \text { for some } \vartheta \in(0,1), \text { with } u_0 \geq 0 \text { and } \int_{\Omega} \ln u_0>-\infty, \quad \text { and that } \\ v_0 \in W^{1, \infty}(\Omega) \text { satisfies } v_0>0 \text { in } \bar{\Omega} \text {. }\end{cases}
\end{align*}
Based on the method of energy estimates, he demonstrated the global weak solution to the initial-boundary value problem of \eqref{-A1-3}. Furthermore, the solution converges to equilibrium within a defined topological space. Later, Li and Winkler \cite{2022-CPAA-LiWinklera} showed the existence of global weak solution $(u,v)$ to system \eqref{-A1-3}, and there exists $u_{\infty} \in C^0(\bar{\Omega})$ such that
$u(\cdot, t) \rightarrow u_{\infty}$ in $L^{\infty}(\Omega)$ and $v(\cdot, t)\rightarrow 0$ in $L^{\infty}(\Omega)$ as $t \rightarrow \infty$. In two-dimensional space, Winkler \cite{2022-NARWA-Winkler} proved that the model \eqref{-A1-1} has the global weak solution for appropriately small initial data. In the higher dimensional setting, for a variant of  \eqref{-A1-1}
\begin{align*}
\begin{cases}u_t=\nabla \cdot(u v \nabla u)-\chi \nabla \cdot\left(u^\alpha v \nabla v\right)+\ell u v, & x \in \Omega, t>0, \\ v_t=\Delta v-u v, & x \in \Omega, t>0,\end{cases}
\end{align*}
in a smoothly bounded convex domain $\Omega \subset \mathbb{R}^n$, with zero-flux boundary conditions, where $\alpha>0$, $\chi>0$ and $\ell \geq 0$, it is shown in \cite{2022-JDE-Li}, the system admits a global weak solution when either $\alpha \in \left(1, \frac{3}{2}\right)$ and $n=2$ or $\alpha \in \left(\frac{7}{6}, \frac{13}{9}\right)$ and $n=3$. In a recent paper \cite{2024-JDE-Winkler}, a more general variant of \eqref{-A1-1}
\begin{align*}
\begin{cases}u_t=\nabla \cdot(u v \nabla u)-\nabla \cdot(S(u) v \nabla v)+\ell u v, & x \in \Omega, t>0, \\ v_t=\Delta v-u v, & x \in \Omega, t>0,\end{cases}
\end{align*} 
was studied, and the global existence of bounded weak solutions was proven in the two-dimensional case when either $S \in C^1([0, \infty))$ satisfies $\limsup _{\xi \rightarrow \infty} \frac{S(\xi)}{\xi^\alpha}<\infty$ with some $\alpha<2$ and the initial data are reasonably regular but arbitrary large, or $S(\xi)=\xi^2$ for $\xi \geq 0$ and the initial data are such that $v_0$ satisfies an appropriate smallness condition.

%%%%%%%%%%%%%%%%%%%%%%%%%%%%%%%%%%%%%%%%%%%%%%%%%%%%%%%%%%%%%%%%%%%%%%%%%%%%%%%%%%%%%%%%

In the past decade, the following relatively simpler nutrient chemotaxis system      
\begin{equation}\label{-A31}
\begin{cases}
u_t=\Delta u-\nabla \cdot(u \nabla v), & x \in \Omega, t>0, \\ v_t=\Delta v-u v, & x \in \Omega, t>0
\end{cases}
\end{equation}
has been considered extensively on the global existence of solutions. It is proven that \eqref{-A31} with homogeneous Neumann boundary conditions in bounded two-dimensional domains admits global classical solutions \cite{2012-CPDE-Winkler} and in three-dimensional bounded domains has global weak solutions that eventually become smooth \cite{2012-jde-taoyou}. When $\left\|v_0 \right\|_{L^{\infty}(\Omega)}$ is sufficiently small and $n \geq 2$, Tao \cite{2011-JMAA-Tao} obtained global bounded classical solutions. In the higher dimensional setting, Wang and Li \cite{2019-EJDE-WangLi} showed that this model possesses at least one
global renormalized solution. We refer to \cite{2016-CVPDE-CaoLankeit, 2018-JDE-Winklera} and the survey \cite{2015-MMMAS-BellomoBellouquidTaoWinkler,2023-SAM-LankeitWinkler} for more related results.

Motivated by the above works, it is, therefore, the objective of this paper to investigate the global existence of bounded weak solutions for a doubly degenerate nutrient taxis system \eqref{-A1-1} without any smallness restriction on initial data in two-dimensional setting. We now state our main results.

\textbf{Main results.} In this paper, we shall be concerned with the initial-boundary value problem
\begin{equation}\label{-A1}
\begin{cases}u_t=\nabla \cdot(u v \nabla u)-\chi \nabla \cdot\left(u^{2}v \nabla v\right)+ \ell u v, & x \in \Omega, t>0, \\ v_t=\Delta v-u v, & x \in \Omega, t>0, \\ \left(u v \nabla u-\chi u^{2} v \nabla v\right) \cdot \nu=\nabla v \cdot \nu=0, & x \in \partial \Omega, t>0, \\ u(x, 0)=u_0(x), \quad v(x, 0)=v_0(x), & x \in \Omega \end{cases}
\end{equation}
in a smoothly convex bounded domain $\Omega \subset \mathbb{R}^2$. To state our main result, throughout our analysis, we will assume that initial data satisfy 
\begin{equation}\label{th1-2}
\begin{cases}
u_0 \in W^{1, \infty}(\Omega) \text { is nonnegative with } u_0 \not \equiv 0  \text { and } \int_{\Omega} \ln u_0>-\infty, \quad \text { and }  \\
v_0 \in W^{1, \infty}(\Omega) \text { is positive in } \bar{\Omega} .
\end{cases}
\end{equation}
\begin{thm} \label{th1}
Let $\chi>0$, $\ell \geq 0$ and let $\Omega \subset \mathbb{R}^2$ be a bounded convex domain with smooth boundary. Assume that the initial value $\left(u_0, v_0\right)$ satisfies \eqref{th1-2}. Then there exist functions
\begin{align}\label{5.4-1429}
\begin{cases}
u \in C^{0}(\bar{\Omega} \times[0, \infty)) \quad \text { and } \\
v \in C^0(\bar{\Omega} \times[0, \infty)) \cap C^{2,1}(\bar{\Omega} \times(0, \infty))
\end{cases}
\end{align}
such that $u > 0$ and $v>0$ a.e in $\Omega \times(0, \infty)$, and that $(u, v)$ solves \eqref{-A1} in that in the sense of Definition \ref{def2.1} below, and that 
\begin{align*}
\|u(t)\|_{L^\infty(\Omega)}+\|v(t)\|_{W^{1, \infty}(\Omega)} \leq C \quad \text { for all } t>0.
\end{align*} 
\end{thm}

\vskip 3mm
The remainder of this paper is organized as follows. 
In Section \ref{section2}, we define the global weak solutions of the system \eqref{-A1} and provide some initial findings regarding the local-in-time existence of the problem \eqref{-A1}. In Section \ref{section3}, we give $L^p$ bounds for $u$ and state straightforward consequences.
In Section \ref{section4}, we give the proof of the Theorem \ref{th1}.

%%%%%%%%%%%%%%%%%%%%%%%%%%%%%%%%%%%%%%%%%%%%%%%%%%%%%%%%%%%%%%%%%%%%%%%%%%%%%%%%%%%%%%%%
%By gaining information on the corresponding dissipation rates, we can confirm that with additional constraints on $m$, $n$ and $\alpha$ and singular behavior of $\phi$, global solvability can be ensured in frameworks of standard weak solvability with respect to the first equation in (\ref{-A1}).
%%%%%%%%%%%%%%%%%%%%%%%%%%%%%%%%%%%%%%%%%%%%%%%%%%%%%%%%%%%%%%%%%%%%%%%%%%%%%%%%%%%%%%%%

%%%%%%%%%%%%%%%%%%%%%%%%%%%%%%%%%%%%%%%%%%%%%%%%%%%%%%%%%%%%%%%%%%%%%%%%%%%%%%%%%%%%%%%%

\section{Preliminaries. Classical solutions to regularized problems}\label{section2}

We first define the global weak solutions of the system (\ref{-A1}). 

%%%%%%%%%%%%%%%%%%%%%%%%%%%%%%%%%%%%%%%%%%%%%%%%%%%%%%%%%%%%%%%%%%%%%%%%%%%%%%%%%%%%%%%%
\begin{defn} \label{def2.1}
Let $\chi>0$, $\ell \geq 0$ and $\Omega \subset \mathbb{R}^n$ $(n \geq 1)$ be a bounded domain with smooth boundary. Suppose that $u_0 \in L^1(\Omega)$ and $v_0 \in L^1(\Omega)$ are nonnegative. Then by a global weak solution of (\ref{-A1}) we mean a pair $(u, v)$ of functions satisfying
\begin{equation}\label{-2.1}
\begin{cases}
u \in L_{l o c}^1(\bar{\Omega} \times[0, \infty)) \quad \text { and } \\
v \in L_{l o c}^{\infty}(\bar{\Omega} \times[0, \infty)) \cap L_{l o c}^1\left([0, \infty) ; W^{1,1}(\Omega)\right)
\end{cases}
\end{equation}
and
\begin{equation}\label{-2.2}
u^2 \in L_{l o c}^1\left([0, \infty) ; W^{1,1}(\Omega)\right) \quad \text { and } \quad u^{2} v \nabla v \in L_{l o c}^1\left(\bar{\Omega} \times[0, \infty) ; \mathbb{R}^n\right)
\end{equation}
which are such that
\begin{align}\label{-2.3}
-\int_0^{\infty} \int_{\Omega} u \varphi_t-\int_{\Omega} u_0 \varphi(\cdot, 0)=&\frac{1}{2} \int_0^{\infty} \int_{\Omega} v  \nabla u^{2} \cdot \nabla \varphi
+\chi \int_0^{\infty} \int_{\Omega} u^{2} v \nabla v \cdot \nabla \varphi \nonumber\\
& +\ell \int_0^{\infty} \int_{\Omega} u v \varphi 
\end{align}
for all $\varphi \in C_0^{\infty}(\bar{\Omega} \times[0, \infty))$ fulfilling $\frac{\partial \varphi}{\partial \nu}=0$ on $\partial \Omega \times(0, \infty)$, and that
\begin{align}\label{-2.4}
\int_0^{\infty} \int_{\Omega} v \varphi_t+\int_{\Omega} v_0 \varphi(\cdot, 0)=\int_0^{\infty} \int_{\Omega} \nabla v \cdot \nabla \varphi+\int_0^{\infty} \int_{\Omega} u v \varphi
\end{align}
for all $\varphi \in C_0^{\infty}(\bar{\Omega} \times[0, \infty))$.

\end{defn}
In order to accomplish the construction of weak solutions through a convenient approximation, for $\chi>0$, $\ell \geq 0$ and $\varepsilon \in(0,1)$ we consider the problem
\begin{align}\label{-2.5}
\begin{cases}u_{\varepsilon t}=\nabla \cdot\left(u_{\varepsilon} v_{\varepsilon} \nabla u_{\varepsilon}\right)-\chi \nabla \cdot\left(u_{\varepsilon}^{2} v_{\varepsilon} \nabla v_{\varepsilon}\right)+\ell u_{\varepsilon} v_{\varepsilon}, & x \in \Omega, t>0, \\ v_{\varepsilon t}=\Delta v_{\varepsilon}-u_{\varepsilon} v_{\varepsilon}, & x \in \Omega, t>0, \\ \frac{\partial u_{\varepsilon}}{\partial v}=\frac{\partial v_{\varepsilon}}{\partial v}=0, & x \in \partial \Omega, t>0, \\ u_{\varepsilon}(x, 0)=u_{0 \varepsilon}(x):=u_0(x)+\varepsilon, \quad v_{\varepsilon}(x, 0)=v_{0 \varepsilon}(x):=v_0(x), & x \in \Omega,\end{cases}
\end{align}
which indeed admits a global classical solution.\\

Next, we give the local existence and extension criterion for the system \eqref{-2.5}. The following lemma is proved in a similar way as Lemma 2.2 in  \cite{2021-TAMS-Winkler}, so we omit the details.
\begin{lem}\label{lemma-2.1}
Let $\chi>0$, $\ell \geq 0$ and let $\Omega \subset \mathbb{R}^n$ $(n \geq 1)$ be a bounded domain with smooth boundary, and suppose that (\ref{th1-2}) holds. Then for each $\varepsilon \in(0,1)$, there exist $T_{\max , \varepsilon} \in(0, \infty]$ and at least one pair $\left(u_{\varepsilon}, v_{\varepsilon}\right)$ of functions
\begin{align}\label{-2.6}
\begin{cases}
u_{\varepsilon} \in C^0\left(\bar{\Omega} \times\left[0, T_{\max , \varepsilon}\right)\right) \cap C^{2,1}\left(\bar{\Omega} \times\left(0, T_{\max , \varepsilon}\right)\right) \quad \text{ and } \\
v_{\varepsilon} \in \cap_{q \geq 1} C^0\left(\left[0, T_{\max , \varepsilon}\right) ; W^{1, q}(\Omega)\right) \cap C^{2,1}\left(\bar{\Omega} \times\left(0, T_{\max , \varepsilon}\right)\right)
\end{cases}
\end{align}
which are such that $u_{\varepsilon}>0$ and $v_{\varepsilon}>0$ in $\bar{\Omega} \times\left(0, T_{\max , \varepsilon}\right)$, that $\left(u_{\varepsilon}, v_{\varepsilon}\right)$ solves \eqref{-2.5} in the classical sense, and that
\begin{align}\label{-2.7}
if \text{  } T_{\max , \varepsilon}<\infty, \quad \text { then } \quad \limsup _{t \nearrow T_{\max , \varepsilon}}\left\|u_{\varepsilon}(t)\right\|_{L^{\infty}(\Omega)}=\infty.
\end{align}
In addition, this solution satisfies
\begin{align}\label{-2.8}
\int_{\Omega} u_{0 \varepsilon} \leq \int_{\Omega} u_{\varepsilon}(t) \leq \int_{\Omega} u_{0 \varepsilon}+\ell \int_{\Omega} v_{0 \varepsilon} \quad \text { for all } t \in\left(0, T_{\max , \varepsilon}\right)
\end{align}
and
\begin{align}\label{-2.81}
\int_{\Omega} v_{\varepsilon}(\cdot,t) \leq \int_{\Omega} v_{\varepsilon}\left(\cdot, t_0\right) \quad \text { for all } t_0 \in\left[0, T_{\max , \varepsilon}\right) \text { and any } t \in\left(t_0, T_{\max , \varepsilon}\right)
\end{align}
and
\begin{align}\label{-2.9}
\left\|v_{\varepsilon}(\cdot,t)\right\|_{L^{\infty}(\Omega)} \leq\left\|v_{\varepsilon}\left(\cdot, t_0\right)\right\|_{L^{\infty}(\Omega)} \quad \text { for all } t_0 \in\left[0, T_{\max , \varepsilon}\right) \text { and any } t \in\left(t_0, T_{\max , \varepsilon}\right)
\end{align}
as well as
\begin{align}\label{-2.10}
\int_{t_0}^{T_{\max , \varepsilon}} \int_{\Omega} u_{\varepsilon} v_{\varepsilon} \leq \int_{\Omega} v_{\varepsilon}\left(\cdot, t_0\right) \quad \text { for all } t_0 \in\left[0, T_{\max , \varepsilon}\right).
\end{align}
\end{lem}

%%%%%%%%%%%%%%%%%%%%%%%%%%%%%%%%%%%%%%%%%%%%%%%%%%%%%%%%%%%%%%%%%%%%%%%%%%%%%%%%%%%%%%%%

%%%%%%%%%%%%%%%%%%%%%%%%%%%%%%%%%%%%%%%%%%%%%%%%%%%%%%%%%%%%%%%%%%%%%%%%%%%%%%%%%%%%%%%%

%%%%%%%%%%%%%%%%%%%%%%%%%%%%%%%%%%%%%%%%%%%%%%%%%%%%%%%%%%%%%%%%%%%%%%%%%%%%%%%%%%%%%%%%

%%%%%%%%%%%%%%%%%%%%%%%%%%%%%%%%%%%%%%%%%%%%%%%%%%%%%%%%%%%%%%%%%%%%%%%%%%%%%%%%%%%%%%%%

%%%%%%%%%%%%%%%%%%%%%%%%%%%%%%%%%%%%%%%%%%%%%%%%%%%%%%%%%%%%%%%%%%%%%%%%%%%%%%%%%%%%%%%%

\section{Degeneracy control via functional inequalities. $L^p$ bounds for u}\label{section3}
Due to the boundedness of the mass of cells and the fact that is indicated by the second equation of \eqref{-2.5}, the nutrients are depleted by cells the following boundedness properties, already inherent in the system \eqref{-2.5} are of great importance.

\begin{lem}\label{lem-3.3}
Let $n \geq 1$, and assume that (\ref{th1-2}) holds. Then for all $K > 0$ there exists $C(K) > 0$ such that
\begin{equation}\label{-3.4} 
\int_0^{T_{\max , \varepsilon}} \int_{\Omega} v_{\varepsilon}\left|\nabla v_{\varepsilon}\right|^2 \leq C(K)
\end{equation}
for all $\varepsilon \in(0,1)$.
\end{lem}
\begin{proof}
Let $\varepsilon \in(0,1)$. Multiplying the second equation in (\ref{-2.5}) by $v_{\varepsilon}^{2}$ and integrating by parts, we see that
\begin{align*}
\frac{1}{3} \frac{d}{d t} \int_{\Omega} v_{\varepsilon}^3=-2 \int_{\Omega} v_{\varepsilon}\left|\nabla v_{\varepsilon}\right|^2-\int_{\Omega} u_{\varepsilon} v_{\varepsilon}^3 \quad \text { for all } t \in\left(0, T_{\max , \varepsilon}\right).
\end{align*}
Based on \eqref{-2.9}, an integration of this implies that
\begin{align*}
\frac{1}{3} \int_{\Omega} v_{\varepsilon}^3(t)+2 \int_0^t \int_{\Omega} v_{\varepsilon}\left|\nabla v_{\varepsilon}\right|^2 & \leq \frac{1}{3} \int_{\Omega} v_{\varepsilon}^3(\cdot, 0) \\
& \leq \frac{|\Omega| \left\|v_0\right\|_{L^{\infty}(\Omega)}^3}{3} \quad \text { for all } t \in\left(0, T_{\max , \varepsilon}\right). 
\end{align*}
In view of the above inequality, we obtain \eqref{-3.4}. 
\end{proof}

\begin{lem}\label{n-lem-3.2a}
Let $K>0$. Then there exists $C(K)>0$ such that if $u_0$ and $v_0$ satisfy \eqref{th1-2} and \begin{align}\label{n-3.1}
\int_{\Omega} \ln u_0 \geq-K \quad \text { and } \quad \int_{\Omega} u_0 \leq K
\end{align}
as well as
\begin{align}\label{n-3.2}
\int_{\Omega} v_0 \leq K \quad \text { and } \quad \int_{\Omega} |\nabla v_{0 }|^2 \leq K,
\end{align}
then we have
\begin{align}\label{n-3.4}
\int_0^{T_{\max , \varepsilon}} \int_{\Omega} \frac{v_{\varepsilon}}{u_{\varepsilon}}|\nabla  u_{\varepsilon }|^2 \leq C(K) \quad \text { for all } \varepsilon \in(0,1)
\end{align}
and
\begin{align}\label{n-3.5} 
\int_0^{T_{\max , \varepsilon}} \int_{\Omega} v_{\varepsilon} \leq C(K) \quad \text { for all } \varepsilon \in(0,1).
\end{align}
\end{lem}
\begin{proof}
Using \eqref{-2.5} and integrations by parts, we obtain
\begin{align}\label{n-3.6}
&\frac{d}{d t}\left\{-\int_{\Omega} \ln u_{\varepsilon}+  \frac{1}{2} \int_{\Omega} |\nabla v_{\varepsilon }|^2\right\} \nonumber\\
= & -\int_{\Omega} \frac{1}{u_{\varepsilon}} \cdot \left\{\nabla \cdot \left(u_{\varepsilon} v_{\varepsilon} \nabla u_{\varepsilon }-u_{\varepsilon}^2 v_{\varepsilon} \nabla v_{\varepsilon }\right)+u_{\varepsilon} v_{\varepsilon}\right\} \nonumber\\
& +\int_{\Omega} \nabla v_{\varepsilon } \cdot \nabla \left( \Delta  v_{\varepsilon }-u_{\varepsilon} v_{\varepsilon}\right) \nonumber\\
= & -\int_{\Omega} \frac{v_{\varepsilon}}{u_{\varepsilon}} |\nabla u_{\varepsilon }|^2+\int_{\Omega} v_{\varepsilon} \nabla u_{\varepsilon }\cdot \nabla v_{\varepsilon }-\int_{\Omega} v_{\varepsilon} \nonumber\\
& -\int_{\Omega}|\Delta v_{\varepsilon }|^2-\int_{\Omega} v_{\varepsilon} \nabla u_{\varepsilon }\cdot \nabla v_{\varepsilon }-\int_{\Omega} u_{\varepsilon} | \nabla v_{\varepsilon }|^2 \nonumber\\
\leq & -\int_{\Omega} \frac{v_{\varepsilon}}{u_{\varepsilon}} | \nabla u_{\varepsilon }|^2-\int_{\Omega} v_{\varepsilon} \quad \text { for all } t \in(0, T_{\max , \varepsilon}) \text { and } \varepsilon \in(0,1).
\end{align}
On integration in time, in view of \eqref{n-3.1} and \eqref{n-3.2} this implies that
\begin{align}\label{n-3.7}
-\int_{\Omega} \ln u_{\varepsilon}(t)+\int_0^{T_{\max , \varepsilon}} \int_{\Omega} \frac{v_{\varepsilon}}{u_{\varepsilon}} | \nabla u_{\varepsilon }|^2+\int_0^{T_{\max , \varepsilon}} \int_{\Omega} v_{\varepsilon} & \leq-\int_{\Omega} \ln u_0+\frac{1}{2} \int_{\Omega} | \nabla v_{0 }|^2 \nonumber\\
& \leq \frac{3}{2} K \quad \text { for all } \varepsilon \in(0,1).
\end{align}
According to \eqref{-2.8}, \eqref{n-3.1}, \eqref{n-3.2} and $\ln \xi \leq \xi$ for all $\xi>0$, we see that
\begin{align}\label{513-943}
\int_{\Omega} \ln u_{\varepsilon}(t) \leq \int_{\Omega} u_{\varepsilon}(t) \leq \int_{\Omega} u_0+\ell \int_{\Omega} v_0 \leq (\ell +1) K.
\end{align}
From \eqref{n-3.7}, we infer that
\begin{align*}
\int_0^{T_{\max , \varepsilon}} \int_{\Omega} \frac{v_{\varepsilon}}{u_{\varepsilon}}|\nabla u_{\varepsilon }|^2+\int_0^{T_{\max , \varepsilon}} \int_{\Omega} v_{\varepsilon} \leq (\ell +\frac{5}{2}) K \quad \text { for all } \varepsilon \in(0,1)
\end{align*}
and that thus also \eqref{n-3.4} and \eqref{n-3.5} hold.
\end{proof}

We now use the estimate \eqref{n-3.4} to derive the following additional
information on decay in the second solution component.

\begin{lem}\label{n-lem-3.3a}
Let $K>0$. Then there exists $C(K)>0$ such that if $u_0$ and $v_0$ satisfy \eqref{th1-2} and
\begin{align}\label{n-3.8}
\int_{\Omega} \frac{|\nabla v_{0 }|^2}{v_0} \leq K,
\end{align}
then
\begin{align}\label{n-3.9}
\int_0^{T_{\max , \varepsilon}} \int_{\Omega} \frac{u_{\varepsilon}}{v_{\varepsilon}}|\nabla  v_{\varepsilon }|^2 \leq C(K) \quad \text { for all } \varepsilon \in(0,1)
\end{align}
and
\begin{align}\label{n-3.10}
\int_0^{T_{\max , \varepsilon}} \int_{\Omega} \frac{|\nabla v_{\varepsilon }|^4}{v_{\varepsilon}^3} \leq C(K) \quad \text { for all } \varepsilon \in(0,1).
\end{align}
\end{lem}
\begin{proof}
From the regularity properties of $u_{\varepsilon}$ and $v_{\varepsilon}$ asserted by Lemma \ref{lemma-2.1}, in view of standard parabolic Schauder theory, it follows that $ v_{\varepsilon }$ belongs to $C^{2,1}(\bar{\Omega} \times(0, \infty))$ and satisfies the accordingly differentiated version of the second equation in \eqref{-2.5}. Thanks to the strict positivity of $v_{\varepsilon}$ in $\bar{\Omega} \times(0, \infty)$, 
\begin{align*}
\nabla v_{\varepsilon} \cdot \nabla \Delta v_{\varepsilon}=\frac{1}{2} \Delta\left|\nabla v_{\varepsilon}\right|^2-\left|D^2 v_{\varepsilon}\right|^2
\end{align*}
and
\begin{align*}
\left|D^2 \ln v_{\varepsilon}\right|^2=\frac{1}{v_{\varepsilon}^2}\left|D^2 v_{\varepsilon}\right|^2-\frac{2}{v_{\varepsilon}^3} \nabla v_{\varepsilon} \cdot\left(D^2 v_{\varepsilon} \cdot \nabla v_{\varepsilon}\right)+\frac{1}{v_{\varepsilon}^4}\left|\nabla v_{\varepsilon}\right|^4,
\end{align*}
we therefore integrate by parts to compute
\begin{align}\label{n-3.11}
\frac{1}{2} \frac{d}{d t} \int_{\Omega} \frac{1}{v_{\varepsilon}}\left|\nabla v_{\varepsilon}\right|^2= & \int_{\Omega} \frac{1}{v_{\varepsilon}} \nabla v_{\varepsilon} \cdot \nabla\left\{\Delta v_{\varepsilon}-u_{\varepsilon} v_{\varepsilon}\right\}-\frac{1}{2} \int_{\Omega} \frac{1}{v_{\varepsilon}^2}\left|\nabla v_{\varepsilon}\right|^2 \cdot\left\{\Delta v_{\varepsilon}-u_{\varepsilon} v_{\varepsilon}\right\} \nonumber\\
= & -\int_{\Omega} \frac{1}{v_{\varepsilon}}\left|D^2 v_{\varepsilon}\right|^2+\frac{1}{2} \int_{\Omega} \frac{1}{v_{\varepsilon}} \Delta\left|\nabla v_{\varepsilon}\right|^2-\frac{1}{2} \int_{\Omega} \frac{1}{v_{\varepsilon}^2}\left|\nabla v_{\varepsilon}\right|^2 \Delta v_{\varepsilon} \nonumber\\
& -\int_{\Omega} \frac{1}{v_{\varepsilon}} \nabla v_{\varepsilon} \cdot \nabla \left(u_{\varepsilon} v_{\varepsilon}\right)+\frac{1}{2} \int_{\Omega} \frac{u_{\varepsilon}}{v_{\varepsilon}}\left|\nabla v_{\varepsilon}\right|^2 \nonumber\\
= & -\int_{\Omega} \frac{1}{v_{\varepsilon}}\left|D^2 v_{\varepsilon}\right|^2+2 \int_{\Omega} \frac{1}{v_{\varepsilon}^2} \nabla v_{\varepsilon} \cdot\left(D^2 v_{\varepsilon} \cdot \nabla v_{\varepsilon}\right)-\int_{\Omega} \frac{1}{v_{\varepsilon}^3}\left|\nabla v_{\varepsilon}\right|^4 \nonumber\\
& +\frac{1}{2} \int_{\partial \Omega} \frac{1}{v_{\varepsilon}} \cdot \frac{\partial\left|\nabla v_{\varepsilon}\right|^2}{\partial \nu} \nonumber\\
& -\int_{\Omega} \nabla u_{\varepsilon} \cdot \nabla v_{\varepsilon}  -\frac{1}{2} \int_{\Omega} \frac{u_{\varepsilon}}{v_{\varepsilon}}\left|\nabla v_{\varepsilon}\right|^2 \nonumber\\
= & -\int_{\Omega} v_{\varepsilon}\left|D^2 \ln v_{\varepsilon}\right|^2+\frac{1}{2} \int_{\partial \Omega} \frac{1}{v_{\varepsilon}} \cdot \frac{\partial\left|\nabla v_{\varepsilon}\right|^2}{\partial \nu} \nonumber\\
&-\int_{\Omega} \nabla u_{\varepsilon} \cdot \nabla v_{\varepsilon}  -\frac{1}{2} \int_{\Omega} \frac{u_{\varepsilon}}{v_{\varepsilon}}\left|\nabla v_{\varepsilon}\right|^2. 
\end{align}
for all $t \in\left(0, T_{\max , \varepsilon}\right)$ and $\varepsilon \in(0,1)$. By \cite[Lemma 3.3]{2012-CPDE-Winkler} and \cite[Lemma 3.4]{2022-NARWA-Winkler}, we can find $c_1>0$ such that 
\begin{equation}\label{n-3.12}
\int_{\Omega} v_{\varepsilon}\left|D^2 \ln v_{\varepsilon}\right|^2 \geq c_1 \int_{\Omega} \frac{1}{v_{\varepsilon}}\left|D^2 v_{\varepsilon}\right|^2+c_1 \int_{\Omega} \frac{1}{v_{\varepsilon}^3}\left|\nabla v_{\varepsilon}\right|^4.
\end{equation}
By Young's inequality,
\begin{align}\label{n-3.14}
-\int_{\Omega} \nabla u_{\varepsilon} \cdot \nabla v_{\varepsilon} 
\leq \frac{1}{4}\int_{\Omega} \frac{u_{\varepsilon}}{v_{\varepsilon}}|\nabla v_{\varepsilon }|^2+\int_{\Omega} \frac{v_{\varepsilon}}{u_{\varepsilon}} |\nabla u_{\varepsilon }|^2 \quad \text { for all } t \in(0, T_{\max , \varepsilon}) \text { and } \varepsilon \in(0,1).
\end{align}
From \eqref{n-3.11}-\eqref{n-3.14} and using that $\frac{\partial|\nabla \varphi|^2}{\partial v} \leq 0$ on $\partial \Omega$ by convexity of $\Omega$ (\cite{1980-ARMA-Lions}), we obtain that
\begin{align}\label{n-3.15}
& \frac{d}{d t} \int_{\Omega} \frac{1}{v_{\varepsilon}}\left|\nabla v_{\varepsilon}\right|^2+c_1 \int_{\Omega} \frac{1}{v_{\varepsilon}}\left|D^2 v_{\varepsilon}\right|^2+c_1 \int_{\Omega} \frac{1}{v_{\varepsilon}^3}\left|\nabla v_{\varepsilon}\right|^4+\frac{1}{2} \int_{\Omega} \frac{u_{\varepsilon}}{v_{\varepsilon}} |\nabla v_{\varepsilon }|^2 \nonumber\\
\leq &  2 \int_{\Omega} \frac{v_{\varepsilon}}{u_{\varepsilon}} |\nabla u_{\varepsilon }|^2 \quad \text { for all } t \in(0, T_{\max , \varepsilon}) \text { and } \varepsilon \in(0,1).
\end{align}
A direct integration in \eqref{n-3.15} thereafter shows that
\begin{align*}
\int_{\Omega} \frac{1}{v_{\varepsilon}}\left|\nabla v_{\varepsilon}\right|^2+c_1 \int_0^{T_{\max , \varepsilon}} \int_{\Omega} \frac{1}{v_{\varepsilon}}\left|D^2 v_{\varepsilon}\right|^2&+c_1 \int_0^{T_{\max , \varepsilon}} \int_{\Omega} \frac{1}{v_{\varepsilon}^3}\left|\nabla v_{\varepsilon}\right|^4+\frac{1}{2}\int_0^{T_{\max , \varepsilon}}  \int_{\Omega} \frac{u_{\varepsilon}}{v_{\varepsilon}} |\nabla v_{\varepsilon }|^2 \\
& \leq \int_{\Omega} \frac{|\nabla v_{0 }|^2}{v_0}+ \int_0^{T_{\max , \varepsilon}} \int_{\Omega} \frac{v_{\varepsilon}}{u_{\varepsilon}}|\nabla  u_{\varepsilon }|^2 \quad \text { for all } \varepsilon \in(0,1)
\end{align*}
and that thus \eqref{n-3.9} and \eqref{n-3.10} hold due to Lemma \ref{n-lem-3.2a} and \eqref{n-3.8}.
\end{proof}

We present the relationship between $\int_{\Omega} \varphi^{p +1} \psi$ and $\int_{\Omega} \varphi^{p}$, which plays an important role in proving Theorem \ref{th1}. Next, we can derive a space-time $L^{2}$ bound for $u_{\varepsilon}$, weighted by the factor $v_{\varepsilon}$ using this relationship and the estimation obtained from Lemmas \ref{n-lem-3.3a}, \ref{n-lem-3.2a}, \ref{lem-3.3}, which  is the key to proving Lemma \ref{4.10-lem-1}. Unlike in most precedents, however, thanks to \eqref{-2.10}, \eqref{n-3.2} and \eqref{n-3.9}, the information thereby generated will here even include corresponding integrability over the whole existence interval, thus implicitly containing certain decay information.

\begin{lem}\label{lemma-4.23-23:32}
Let $\Omega \subset \mathbb{R}^2$ be a bounded domain with smooth boundary, and let $p \geq 1$. Then one can find $C\left(p \right)>0$ such that for each $\varphi \in C^1(\bar{\Omega})$ and $\psi \in C^1(\bar{\Omega})$ fulfilling $\varphi \geq 0$ and $\psi>0$ in $\bar{\Omega}$,
we have
\begin{align*}
\int_{\Omega} \varphi^{p +1} \psi \leq  C(p)\cdot \left\{\int_{\Omega} \frac{\varphi}{\psi}|\nabla \psi|^2+\int_{\Omega} \frac{\psi}{\varphi}|\nabla \varphi|^2+ \int_{\Omega} \varphi \psi  \right\}\cdot \int_{\Omega} \varphi^p .
\end{align*}
\end{lem}
\begin{proof}
According to the Sobolev inequality in the two-dimensional domain $\Omega$, we can find $c_1>0$ such that
\begin{align}\label{5.2-2238}
\int_{\Omega} \rho^2 \leq c_1\|\nabla \rho\|_{L^1(\Omega)}^2+c_1\|\rho\|_{L^{1}(\Omega)}^2 \quad \text { for all } \rho \in C^1(\bar{\Omega}),
\end{align}
which for fixed nonnegative $\varphi \in C^1(\bar{\Omega})$ and positive $\psi \in C^1(\bar{\Omega})$ we apply to $\rho:=\varphi^{\frac{p +1}{2}} \psi^{\frac{1}{2}}$ to infer that
\begin{align}\label{5.2-2304}
\int_{\Omega} \varphi^{p +1} \psi \leq & c_1 \cdot\left\{\int_{\Omega}\left|\frac{p +1}{2} \varphi^{\frac{p -1}{2}} \psi^{\frac{1}{2}} \nabla \varphi+\frac{1}{2} \varphi^{\frac{p +1}{2}} \psi^{-\frac{1}{2}} \nabla \psi\right|\right\}^2 \nonumber\\
& +c_1\cdot\left\{\int_{\Omega}\varphi^{\frac{p +1}{2}} \psi^{\frac{1}{2}}\right\}^{2}\nonumber\\
\leq & \frac{(p +1)^2 c_1}{2} \cdot\left\{\int_{\Omega} \varphi^{\frac{p -1}{2}} \psi^{\frac{1}{2}}|\nabla \varphi|\right\}^2+\frac{c_1}{2} \cdot\left\{\int_{\Omega} \varphi^{\frac{p +1}{2}} \psi^{-\frac{1}{2}}|\nabla \psi|\right\}^{2} \nonumber\\
& +c_1\cdot\left\{\int_{\Omega}\varphi^{\frac{p +1}{2}} \psi^{\frac{1}{2}}\right\}^{2}.
\end{align}
By Cauchy-Schwarz inequality,
\begin{align*}
\left\{\int_{\Omega} \varphi^{\frac{p -1}{2}} \psi^{\frac{1}{2}}|\nabla \varphi|\right\}^2 =\left\{\int_{\Omega} \varphi^{\frac{p}{2}}\cdot \frac{\psi^{\frac{1}{2}}}{\varphi^{\frac{1}{2}}} |\nabla \varphi|\right\}^2 \leq \int_{\Omega} \varphi^p  \cdot \int_{\Omega} \frac{\psi}{\varphi}|\nabla \varphi|^2
\end{align*}
and
\begin{align*}
\left\{\int_{\Omega} \varphi^{\frac{p +1}{2}} \psi^{-\frac{1}{2}}|\nabla \psi|\right\}^2 =
\left\{\int_{\Omega} \varphi^{\frac{p}{2}}\cdot \frac{\varphi^{\frac{1}{2}}}{\psi^{\frac{1}{2}}} |\nabla \psi|\right\}^2
\leq \int_{\Omega} \varphi^p  \cdot \int_{\Omega} \frac{\varphi}{\psi}|\nabla \psi|^2.
\end{align*}
Using H\"{o}lder inequality we see that
\begin{align*}
\left\{\int_{\Omega}\varphi^{\frac{p +1}{2}} \psi^{\frac{1}{2}}\right\}^{2} & 
=\left\{\int_{\Omega} \varphi^{\frac{p }{2}}  \cdot (\varphi \psi)^{\frac{1}{2}}\right\}^{2} \leq \int_{\Omega} \varphi^p  \cdot \int_{\Omega} \varphi \psi.
\end{align*}
The claim therefore results from \eqref{5.2-2304} if we let $C(p):=\max \left\{\frac{(p +1)^2 c_1}{2}, c_1\right\}$.
\end{proof}

As a necessary step for a later argument on the regularity of $u_{\varepsilon}$ in $L^p$, we must establish the following functional inequality as the second of our essential tools. It's important to note that this derivation requires a Sobolev embedding property based on the assumption that the spatial setting is either one- or two-dimensional.

\begin{lem}\label{4.7-lem-2-1}
Let $\Omega \subset \mathbb{R}^2$ be a bounded domain with smooth boundary, and let $p \geq 1$. Then for any $\eta>0$ there exists $C(p, \eta)>0$ such that for any $\varphi \in C^1(\bar{\Omega})$ and $\psi \in C^1(\bar{\Omega})$ fulfilling $\varphi \geq 0$ and $\psi>0$ in $\bar{\Omega}$, for all $\eta>0$ we have
\begin{align}\label{4.7-1}
\int_\Omega \varphi^{p+1} \psi |\nabla \psi|^2
\leq & \eta \int_\Omega \varphi^{p-1} \psi|\nabla \varphi|^2+ \eta \int_{\Omega}\varphi \psi \nonumber\\
& +  C(p, \eta)\cdot \left(\left\|\psi \right\|_{L^{\infty}(\Omega)}+\left\|\psi \right\|^3_{L^{\infty}(\Omega)}\right) \cdot \int_{\Omega} \varphi^{p+1} \psi  \cdot \int_\Omega \frac{|\nabla \psi|^4}{\psi^3} \nonumber\\
&+C(p, \eta)  \cdot \left\|\psi \right\|^4_{L^{\infty}(\Omega)} \cdot\left\{\int_\Omega \varphi\right\}^{2 p+1} \cdot \int_\Omega \frac{|\nabla \psi|^4}{\psi^3} 
\end{align}
\end{lem}
\begin{proof}
According to \eqref{5.2-2238}, we can find $c_1>0$ fulfilling
\begin{align*}
\|\rho\|_{L^2(\Omega)} \leq c_1\|\nabla \rho\|_{L^1(\Omega)}+c_1\|\rho\|_{L^1(\Omega)} \quad \text {  for  } \rho \in W^{1,1}(\Omega).
\end{align*}
By H\"{o}lder's and Young's inequalities, for $p \geq 1$ we imply that
\begin{align*}
c_1\|\rho\|_{L^1(\Omega)} & \leq c_1\|\rho\|_{L^2(\Omega)}^{\frac{2 p}{2 p+1}}\|\rho\|_{L^{\frac{1}{p+1}}(\Omega)}^{\frac{1}{2 p+1}} \\
& =\left\{\frac{1}{2}\|\rho\|_{L^2(\Omega)}\right\}^{\frac{2 p}{2 p+1}} \cdot 2^{\frac{2 p}{2 p+1}} c_1\|\rho\|_{L^{\frac{1}{p+1}}(\Omega)}^{\frac{1}{2 p+1}}\\
& \leq \frac{1}{2}\|\rho\|_{L^2(\Omega)}+2^{2 p} c_1^{2 p+1}(\Omega)\|\rho\|_{L^{\frac{1}{p+1}}(\Omega)} \quad \text { for } \rho \in L^2(\Omega),
\end{align*}
it follows that
\begin{align}\label{4.7-2}
\|\rho\|_{L^2(\Omega)} \leq c_2(p)\|\nabla \rho\|_{L^1(\Omega)}+c_2(p)\|\rho\|_{L^{\frac{1}{p+1}}(\Omega)} \quad \text { for } \rho \in W^{1,1}(\Omega)
\end{align}
with $c_2(p):=\max \left\{2 c_1,\left(2 c_1\right)^{2 p+1}\right\}$. 
Using H\"{o}lder's inequality, we can estimate
\begin{align}\label{4.7-3}
\int_\Omega \varphi^{p+1} \psi |\nabla \psi|^2 \leq\left\{\int_\Omega \frac{|\nabla \psi|^4}{\psi^3}\right\}^{\frac{1}{2}} \cdot\left\{\int_\Omega \varphi^{2 (p+1)} \psi^{5}\right\}^{\frac{1}{2}}.
\end{align}
In view of \eqref{4.7-2}, we can control the second factor according to
\begin{align}\label{4.7-4}
\left\{\int_\Omega \varphi^{2 (p+1)} \psi^{5}\right\}^{\frac{1}{2}} = & \left\|\varphi^{p+1} \psi^{\frac{5}{2}}\right\|_{L^2(\Omega)} \nonumber\\
\leq &  c_2(p) \int_\Omega\left|(p+1) \varphi^{p} \psi^{\frac{5}{2}} \nabla \varphi+\frac{5 }{2} \varphi^{p+1}\psi^\frac{3}{2}\nabla \psi\right|\nonumber\\
&  + c_2(p) \cdot\left\{\int_\Omega \varphi \psi^{\frac{5}{2 (p+1)}}\right\}^{p+1} \nonumber\\
= &  (p+1) c_2(p) \int_\Omega \varphi^{p} \psi^\frac{5}{2}|\nabla \varphi|+\frac{5 c_2(p)}{2} \int_\Omega \varphi^{p+1}\psi^\frac{3}{2}|\nabla \psi| \nonumber\\
& + c_2(p) \cdot\left\{\int_\Omega \varphi \psi^{\frac{5}{2 (p+1)}}\right\}^{p+1}.
\end{align}
Applying H\"{o}lder inequality, we show that
\begin{align*}
(p+1) c_2(p) \int_G \varphi^{p} \psi^\frac{5}{2}|\nabla \varphi| \leq p c_2(p)\left\|\psi\right\|_{L^{\infty}(\Omega)}^\frac{3}{2} \cdot\left\{\int_\Omega \varphi^{p+1}\psi\right\}^{\frac{1}{2}} \cdot\left\{\int_\Omega \varphi^{p-1} \psi|\nabla \varphi|^2\right\}^{\frac{1}{2}}
\end{align*}
and
\begin{align*}
\frac{5 c_2(p)}{2} \int_\Omega \varphi^{p+1}\psi^\frac{3}{2}|\nabla \psi| \leq \frac{5 c_2(p)\left\|\psi\right\|_{L^{\infty}(\Omega)}^\frac{1}{2}}{2} \cdot\left\{\int_\Omega \varphi^{p+1}\psi\right\}^{\frac{1}{2}} \cdot\left\{\int_\Omega \varphi^{p+1} \psi|\nabla \psi|^2\right\}^{\frac{1}{2}}
\end{align*}
as well as
\begin{align*}
c_2(p)\left\{\int_\Omega \varphi \psi^{\frac{5}{2 (p+1)}}\right\}^{p+1} & \leq c_2(p)\left\|\psi \right\|_{L^{\infty}(\Omega)}^2  \cdot\left\{\int_\Omega(\varphi \psi)^{\frac{1}{2 (p+1)}} \cdot \varphi^{\frac{2 p+1}{2 (p+1)}}\right\}^{p+1} \\
& \leq c_2(p) \left\|\psi \right\|_{L^{\infty}(\Omega)}^2 \cdot\left\{\int_\Omega \varphi\right\}^{\frac{2 p+1}{2}} \cdot\left\{\int_\Omega \varphi \psi\right\}^{\frac{1}{2}}.
\end{align*}
Substituting \eqref{4.7-4} into \eqref{4.7-3} and Young’s inequality, we infer that for all $\eta > 0$,
\begin{align*}
\int_\Omega \varphi^{p+1} \psi |\nabla \psi|^2 \leq & p c_2(p)\left\|\psi\right\|_{L^{\infty}(\Omega)}^\frac{3}{2} \cdot\left\{\int_\Omega \frac{|\nabla \psi|^4}{\psi^3}\right\}^{\frac{1}{2}} \cdot\left\{\int_\Omega \varphi^{p+1}\psi\right\}^{\frac{1}{2}} \cdot\left\{\int_\Omega \varphi^{p-1} \psi|\nabla \varphi|^2\right\}^{\frac{1}{2}}\\
& +\frac{5 c_2(p)\left\|\psi\right\|_{L^{\infty}(\Omega)}^\frac{1}{2}}{2} \cdot \left\{\int_\Omega \frac{|\nabla \psi|^4}{\psi^3}\right\}^{\frac{1}{2}}\cdot\left\{\int_\Omega \varphi^{p+1}\psi\right\}^{\frac{1}{2}} \cdot\left\{\int_\Omega \varphi^{p+1} \psi|\nabla \psi|^2\right\}^{\frac{1}{2}} \\
& + c_2(p)\left\|\psi \right\|_{L^{\infty}(\Omega)}^2 \cdot\left\{\int_\Omega \frac{|\nabla \psi|^4}{\psi^3}\right\}^{\frac{1}{2}} \cdot\left\{\int_\Omega \varphi\right\}^{\frac{2 p+1}{2}} \cdot\left\{\int_\Omega \varphi \psi\right\}^{\frac{1}{2}} \\
\leq &   \frac{\eta}{2} \int_\Omega \varphi^{p-1} \psi|\nabla \varphi|^2+\frac{p^2 c_2^2(p)\left\|\psi\right\|_{L^{\infty}(\Omega)}^3}{2 \eta} \cdot \int_\Omega \varphi^{p+1} \psi \cdot \int_\Omega \frac{|\nabla \psi|^4}{\psi^3} \\
& +\frac{1}{2} \int_\Omega \varphi^{p+1} \psi |\nabla \psi|^2+\frac{25 c_2^2(p)\left\|\psi\right\|_{L^{\infty}(\Omega)}}{8} \cdot \int_\Omega \varphi^{p+1} \psi \cdot \int_\Omega \frac{|\nabla \psi|^4}{\psi^3}\\
& +\frac{\eta}{2} \int_\Omega \varphi \psi+\frac{c_2^2(p) \left\|\psi \right\|_{L^{\infty}(\Omega)}^4}{2 \eta} \cdot\left\{\int_\Omega \varphi\right\}^{2 p+1} \cdot \int_\Omega \frac{|\nabla \psi|^4}{\psi^3}.
\end{align*}
\end{proof}

It turns out that \eqref{-2.5} admits the following energy-like functional, which forms the core of our analysis.
\begin{lem}\label{lem-513-1912}
There exists $b>0$ such that
\begin{align}\label{512-1156}
\frac{d}{d t}\left\{4b \int_{\Omega} u_{\varepsilon} \ln u_{\varepsilon}+\int_{\Omega} \frac{\left|\nabla v_{\varepsilon}\right|^4}{v_{\varepsilon}^3}\right\} \leq & -b \int_{\Omega} v_{\varepsilon}\left|\nabla u_{\varepsilon}\right|^2- \int_{\Omega} u_{\varepsilon} v_{\varepsilon}^{-3}\left|\nabla v_{\varepsilon}\right|^4\nonumber\\
& -2 \int_{\Omega} v_{\varepsilon}^{-1}\left|\nabla v_{\varepsilon}\right|^2\left|D^2 \ln v_{\varepsilon}\right|^2\nonumber\\
& + 4b \chi^2 \int_{\Omega} u_{\varepsilon}^{2} v_{\varepsilon}\left|\nabla v_{\varepsilon}\right|^2+4 \ell b \int_{\Omega} u^2_{\varepsilon} v_{\varepsilon} \nonumber\\ 
& + 4 \ell b \int_{\Omega} u_{\varepsilon} v_{\varepsilon}
\end{align}
for all $t \in\left(0, T_{\max , \varepsilon}\right)$ and $\varepsilon \in(0,1)$.
\end{lem}
\begin{proof}
Based on \eqref{-2.8}, \eqref{n-3.1} and \eqref{n-3.2}, we have
\begin{align}\label{4.19-1}
\int_{\Omega} u_{\varepsilon}(t) \leq (1+\ell) K :=c_0(K) 
\end{align} 
for all $t \in\left(0, T_{\max , \varepsilon}\right)$ and $\varepsilon \in(0,1)$.
Using Lemma 3.4 with $k=4$ in \cite{2022-DCDSSB-Winkler}, we see that there exists $c_1>0$ such that
\begin{align}\label{4.19-111}
\int_{\Omega} \frac{|\nabla \phi|^6}{\phi^5} \leq c_1 \int_{\Omega} \phi^{-1}|\nabla \phi|^2\left|D^2 \ln \phi\right|^2.
\end{align}
Recall that Lemma 2.3 in \cite{2022-JDE-Li} implies that
\begin{align}\label{4.19-2}
\frac{d}{d t} \int_{\Omega} \frac{\left|\nabla v_{\varepsilon}\right|^4}{v_{\varepsilon}^3}+4 \int_{\Omega} v_{\varepsilon}^{-1}\left|\nabla v_{\varepsilon}\right|^2\left|D^2 \ln v_{\varepsilon}\right|^2 \leq-\int_{\Omega} u_{\varepsilon} v_{\varepsilon}^{-3}\left|\nabla v_{\varepsilon}\right|^4-4 \int_{\Omega} v_{\varepsilon}^{-2}\left|\nabla v_{\varepsilon}\right|^2\left(\nabla u_{\varepsilon} \cdot \nabla v_{\varepsilon}\right)
\end{align}
for all $t \in(0, T_{\max , \varepsilon})$ and $\varepsilon \in(0,1)$. Invoking Young's inequality and \eqref{4.19-111}, we deduce that
\begin{align}\label{4.19-3-a}
-4 \int_{\Omega} v_{\varepsilon}^{-2}\left|\nabla v_{\varepsilon}\right|^2\left(\nabla u_{\varepsilon} \cdot \nabla v_{\varepsilon}\right) & \leq \frac{2}{c_1} \int_{\Omega} \frac{\left|\nabla v_{\varepsilon}\right|^6}{v_{\varepsilon}^5}+2 c_1 \int_{\Omega} v_{\varepsilon}\left|\nabla u_{\varepsilon}\right|^2 \nonumber\\
& \leq 2 \int_{\Omega} v_{\varepsilon}^{-1}\left|\nabla v_{\varepsilon}\right|^2\left|D^2 \ln v_{\varepsilon}\right|^2+2 c_1 \int_{\Omega} v_{\varepsilon}\left|\nabla u_{\varepsilon}\right|^2
\end{align}
for all $t \in(0, T_{\max , \varepsilon})$ and $\varepsilon \in(0,1)$. Multiplying the first equation in \eqref{-2.5} by $1+\ln u_{\varepsilon}$, we use Cauchy-Schwarz inequality to infer that
\begin{align}\label{4.19-31}
\frac{d}{d t} \int_{\Omega} u_{\varepsilon} \ln u_{\varepsilon}&+\int_{\Omega} v_{\varepsilon}\left|\nabla u_{\varepsilon}\right|^2 \nonumber\\
&  = \chi \int_{\Omega} u_{\varepsilon} v_{\varepsilon} \nabla u_{\varepsilon} \cdot \nabla v_{\varepsilon}+\ell \int_{\Omega} u_{\varepsilon} v_{\varepsilon} \ln u_{\varepsilon} +\ell \int_{\Omega} u_{\varepsilon} v_{\varepsilon}\nonumber\\
& \leq   \frac{1}{4} \int_{\Omega} v_{\varepsilon}\left|\nabla u_{\varepsilon}\right|^2+ \chi^2 \int_{\Omega} u_{\varepsilon}^{2} v_{\varepsilon}\left|\nabla v_{\varepsilon}\right|^2 + \ell \int_{\Omega} u_{\varepsilon} v_{\varepsilon} \ln u_{\varepsilon} +\ell \int_{\Omega} u_{\varepsilon} v_{\varepsilon}
\end{align}
for all $t \in(0, T_{\max , \varepsilon})$ and $\varepsilon \in(0,1)$. We conclude from \eqref{4.19-2}-\eqref{4.19-31} and $\ln \xi \leq \xi$ for all $\xi>0$ that
\begin{align*}
& \frac{d}{d t}\left\{4 c_1 \int_{\Omega} u_{\varepsilon} \ln u_{\varepsilon}+\int_{\Omega} \frac{\left|\nabla v_{\varepsilon}\right|^4}{v_{\varepsilon}^3}\right\} \nonumber\\
\leq & -3 c_1 \int_{\Omega} v_{\varepsilon}\left|\nabla u_{\varepsilon}\right|^2 +4c_1 \chi^2 \int_{\Omega} u_{\varepsilon}^{2} v_{\varepsilon}\left|\nabla v_{\varepsilon}\right|^2+4 \ell c_1 \int_{\Omega} u_{\varepsilon} v_{\varepsilon} \ln u_{\varepsilon} \nonumber\\
&+4 \ell c_1 \int_{\Omega} u_{\varepsilon} v_{\varepsilon}-\int_{\Omega} u_{\varepsilon} v_{\varepsilon}^{-3}\left|\nabla v_{\varepsilon}\right|^4-2 \int_{\Omega} v_{\varepsilon}^{-1}\left|\nabla v_{\varepsilon}\right|^2\left|D^2 \ln v_{\varepsilon}\right|^2\nonumber\\
& +2 c_1 \int_{\Omega} v_{\varepsilon}\left|\nabla u_{\varepsilon}\right|^2 \nonumber\\
\leq & -c_1 \int_{\Omega} v_{\varepsilon}\left|\nabla u_{\varepsilon}\right|^2- \int_{\Omega} u_{\varepsilon} v_{\varepsilon}^{-3}\left|\nabla v_{\varepsilon}\right|^4 -2 \int_{\Omega} v_{\varepsilon}^{-1}\left|\nabla v_{\varepsilon}\right|^2\left|D^2 \ln v_{\varepsilon}\right|^2\nonumber\\
& + 4c_1 \chi^2 \int_{\Omega} u_{\varepsilon}^{2} v_{\varepsilon}\left|\nabla v_{\varepsilon}\right|^2+4 \ell c_1 \int_{\Omega} u^2_{\varepsilon} v_{\varepsilon}  + 4 \ell c_1 \int_{\Omega} u_{\varepsilon} v_{\varepsilon} 
\end{align*}
for all $t \in(0, T_{\max , \varepsilon})$ and $\varepsilon \in(0,1)$.
\end{proof}

\begin{lem}\label{4.10-lem-1-xx}
Let $K>0$. Then there exists $C(K)>0$ such that
\begin{align}\label{513-1910}
\int_0^{T_{\max , \varepsilon}} \int_{\Omega} u^2_{\varepsilon} v_{\varepsilon} \leq C(K)  \quad \text { for all } \varepsilon \in(0,1).
\end{align}
\end{lem}
\begin{proof}
An application of \eqref{4.19-1} and Lemma \ref{lemma-4.23-23:32} witn $p =1$ provides $c_1:=c_1(K)>0$ such that
\begin{align}\label{513-1907}
\int_{\Omega} u_{\varepsilon}^{2} v_{\varepsilon} \leq c_1 \left\{\int_{\Omega} \frac{u_{\varepsilon}}{v_{\varepsilon}}|\nabla v_{\varepsilon}|^2+\int_{\Omega} \frac{v_{\varepsilon}}{u_{\varepsilon}}|\nabla u_{\varepsilon}|^2+ \int_{\Omega} u_{\varepsilon} v_{\varepsilon}  \right\} \cdot \int_{\Omega} u_{\varepsilon}\nonumber\\
\leq c_1 c_0(K) \cdot \left\{\int_{\Omega} \frac{u_{\varepsilon}}{v_{\varepsilon}}|\nabla v_{\varepsilon}|^2+\int_{\Omega} \frac{v_{\varepsilon}}{u_{\varepsilon}}|\nabla u_{\varepsilon}|^2+ \int_{\Omega} u_{\varepsilon} v_{\varepsilon}  \right\}
\end{align}
for all $t \in(0, T_{\max , \varepsilon})$ and $\varepsilon \in(0,1)$. Combining \eqref{-2.9}, \eqref{-2.10} with \eqref{n-3.4} and \eqref{n-3.9}, we see that with $c_2:=c_2(K)>0$,
\begin{align*}
\int_0^{T_{\max , \varepsilon}}  \int_{\Omega} \frac{u_{\varepsilon}}{v_{\varepsilon}}|\nabla v_{\varepsilon}|^2+\int_0^{T_{\max , \varepsilon}}  \int_{\Omega} \frac{v_{\varepsilon}}{u_{\varepsilon}}|\nabla u_{\varepsilon}|^2+ \int_0^{T_{\max , \varepsilon}}  \int_{\Omega} u_{\varepsilon} v_{\varepsilon} \leq c_2
\end{align*}
for all $\varepsilon \in(0,1)$. This together with \eqref{513-1907} yields \eqref{513-1910}.
\end{proof}

We are now in a position to derive some boundedness features from the inequality \eqref{512-1156}
\begin{lem}\label{4.10-lem-1}
Let $K>0$. Then there exists $C(K)>0$ such that if $u_0$ and $v_0$ satisfy \eqref{th1-2} and
\begin{align}\label{4.23-n3.8}
\int_{\Omega} u_{0}\ln u_{0}  \leq K
\end{align}
as well as 
\begin{align}\label{4.23-n3.9}
\int_{\Omega} \frac{|\nabla v_{0 }|^4}{v_0^3} \leq K, 
\end{align}
then we have
\begin{align}\label{4.23-n3.101}
\int_{\Omega} \frac{|\nabla v_{\varepsilon}|^4}{v_{\varepsilon}^3} \leq C(K)\quad \text { for all } t \in(0, T_{\max , \varepsilon}) \text { and } \varepsilon \in(0,1)
\end{align}
\begin{align}\label{4.23-n3.10}
\int_0^{T_{\max , \varepsilon}}  \int_{\Omega} v_{\varepsilon}\left|\nabla u_{\varepsilon}\right|^2 \leq C(K) \quad \text { for all } \varepsilon \in(0,1)
\end{align}
\begin{align}\label{4.23-n3.101-q}
\int_0^{T_{\max , \varepsilon}}  \int_{\Omega} u_{\varepsilon} v_{\varepsilon}^{-3}\left|\nabla v_{\varepsilon}\right|^4 \leq C(K) \quad \text { for all } \varepsilon \in(0,1)
\end{align}
and
\begin{align}\label{5.3-13.46} 
\int_0^{T_{\max , \varepsilon}}  \int_{\Omega} \frac{\left|\nabla v_{\varepsilon}\right|^6}{v_{\varepsilon}^5} \leq C(K) \quad \text { for all } \varepsilon \in(0,1).
\end{align}
\end{lem}
\begin{proof}
From Lemma \ref{lem-513-1912}, we know that for all $t \in(0, T_{\max , \varepsilon})$ and $\varepsilon \in(0,1)$
\begin{align*}
\frac{d}{d t}\left\{4 b \int_{\Omega} u_{\varepsilon} \ln u_{\varepsilon} +\int_{\Omega} \frac{\left|\nabla v_{\varepsilon}\right|^4}{v_{\varepsilon}^3}\right\}  & + b \int_{\Omega} v_{\varepsilon}\left|\nabla u_{\varepsilon}\right|^2 
+ \int_{\Omega} u_{\varepsilon} v_{\varepsilon}^{-3}\left|\nabla v_{\varepsilon}\right|^4 + 2 \int_{\Omega} v_{\varepsilon}^{-1}\left|\nabla v_{\varepsilon}\right|^2\left|D^2 \ln v_{\varepsilon}\right|^2\nonumber\\
& \leq  4b \chi^2 \int_{\Omega} u_{\varepsilon}^{2} v_{\varepsilon}\left|\nabla v_{\varepsilon}\right|^2+4 \ell b \int_{\Omega} u^2_{\varepsilon} v_{\varepsilon} + 4 \ell b \int_{\Omega} u_{\varepsilon} v_{\varepsilon}
\end{align*}
holds with a constant $b>0$. Since an application of Lemma \ref{4.7-lem-2-1} to $\eta =\min \left\{\frac{1}{8 \chi^2}, 1\right\}$ provides $c_1:=c_1(K)>0$,
\begin{align*}
4b \chi^2  \int_\Omega u_{\varepsilon}^{2} v_{\varepsilon} |\nabla \psi|^2
\leq & \frac{b}{2} \int_\Omega  v_{\varepsilon}|\nabla u_{\varepsilon}|^2+ \int_{\Omega} u_{\varepsilon} v_{\varepsilon} \nonumber\\
& +  c_1\cdot \left\{\left\|v_{\varepsilon} \right\|_{L^{\infty}(\Omega)}+\left\|v_{\varepsilon} \right\|^3_{L^{\infty}(\Omega)}\right\} \cdot \int_{\Omega} u_{\varepsilon}^{2} v_{\varepsilon}  \cdot \int_\Omega \frac{|\nabla v_{\varepsilon}|^4}{v_{\varepsilon}^3} \nonumber\\
&+c_1 \cdot \left\|v_{\varepsilon} \right\|^4_{L^{\infty}(\Omega)} \cdot\left\{\int_\Omega u_{\varepsilon}\right\}^{3} \cdot \int_\Omega \frac{|\nabla v_{\varepsilon}|^4}{v_{\varepsilon}^3} \quad \text { for all } t \in(0, T_{\max , \varepsilon}) \text { and } \varepsilon \in(0,1),
\end{align*}
thanks to \eqref{-2.9} and \eqref{4.19-1} this entails that
\begin{align}\label{4.24-2050}
\frac{d}{d t}\left\{4 b \int_{\Omega} u_{\varepsilon} \ln u_{\varepsilon} +\int_{\Omega} \frac{\left|\nabla v_{\varepsilon}\right|^4}{v_{\varepsilon}^3}\right\} & + \frac{b}{2} \int_{\Omega} v_{\varepsilon}\left|\nabla u_{\varepsilon}\right|^2 + \int_{\Omega} u_{\varepsilon} v_{\varepsilon}^{-3}\left|\nabla v_{\varepsilon}\right|^4 + 2 \int_{\Omega} v_{\varepsilon}^{-1}\left|\nabla v_{\varepsilon}\right|^2\left|D^2 \ln v_{\varepsilon}\right|^2\nonumber\\
& \leq  c_1\cdot \left(\left\|v_{0} \right\|_{L^{\infty}(\Omega)}+\left\|v_{0} \right\|^3_{L^{\infty}(\Omega)}\right) \cdot \int_{\Omega} u_{\varepsilon}^{2} v_{\varepsilon}  \cdot \int_\Omega \frac{|\nabla v_{\varepsilon}|^4}{v_{\varepsilon}^3} \nonumber\\
&+c_1 \cdot \left\|v_{0} \right\|^4_{L^{\infty}(\Omega)} \cdot c^3_0(K) \cdot \int_\Omega \frac{|\nabla v_{\varepsilon}|^4}{v_{\varepsilon}^3}\nonumber\\
& +4 \ell b \int_{\Omega} u^2_{\varepsilon} v_{\varepsilon} + (4 \ell b+1) \int_{\Omega} u_{\varepsilon} v_{\varepsilon} 
\end{align}
for all $t \in(0, T_{\max , \varepsilon})$ and $\varepsilon \in(0,1)$.  For each $\varepsilon \in(0,1)$, the functions given by
\begin{align*}
f_{\varepsilon}(t):= 4b \int_{\Omega} u_{\varepsilon} \ln u_{\varepsilon}+\int_{\Omega} \frac{\left|\nabla v_{\varepsilon}\right|^4}{v_{\varepsilon}^3}
\end{align*}
as well as
\begin{align*}
m_{\varepsilon}(t):=c_1 \cdot \left\|v_{0} \right\|^4_{L^{\infty}(\Omega)} \cdot c^3_0(K) \cdot \int_\Omega \frac{|\nabla v_{\varepsilon}|^4}{v_{\varepsilon}^3} +4 \ell b \int_{\Omega} u^2_{\varepsilon} v_{\varepsilon} + (4 \ell b+1) \int_{\Omega} u_{\varepsilon} v_{\varepsilon}  
\end{align*}
and
\begin{align*}
k_{\varepsilon}(t):=& c_1\cdot \left(\left\|v_{0} \right\|_{L^{\infty}(\Omega)}+\left\|v_{0} \right\|^3_{L^{\infty}(\Omega)}\right) \cdot \int_{\Omega} u_{\varepsilon}^{2} v_{\varepsilon}
\end{align*}
for all $t \in(0, T_{\max , \varepsilon})$, thus satisfy
\begin{align*}
f_{\varepsilon}^{\prime}(t) \leq m_{\varepsilon}(t) + k_{\varepsilon}(t) f_{\varepsilon}(t) \quad \text { for all } t \in\left(0, T_{\max , \varepsilon} \right), 
\end{align*}
which upon an ODE comparison argument implies that
\begin{align}\label{513-2009}
f_{\varepsilon}(t) & \leq f_{\varepsilon}(0) e^{\int_0^t k_{\varepsilon}(s) d s}+\int_0^t e^{\int_s^t k_{\varepsilon}(\sigma) d \sigma} m_{\varepsilon}(s) d s
\quad \text { for all } t \in\left(0, T_{\max , \varepsilon}\right). 
\end{align}
Since 
\begin{align*}
\int_s^t k_{\varepsilon}(\sigma) d \sigma \leq  c_2(K) \quad \text { for all } t \in\left(0, T_{\max , \varepsilon}\right), s \in[0, t) \text {  and   } \varepsilon \in(0,1)
\end{align*}
by Lemma \ref{4.10-lem-1-xx}, and since
\begin{align*}
\int_0^t m_{\varepsilon}(s) d s \leq  c_3(K) \quad \text { for all } t \in\left(0, T_{\max , \varepsilon}\right) \text { and } \varepsilon \in(0,1)
\end{align*}
due to \eqref{-2.10}, \eqref{n-3.10} and Lemma \ref{4.10-lem-1-xx}, from \eqref{513-2009}, \eqref{4.23-n3.8} and \eqref{4.23-n3.9} we thus obtain 
\begin{align*}
f_{\varepsilon}(t) & \leq  (4b+1)K \cdot e^{c_2(K)}
+c_3(K) e^{c_2(K)}\quad \text { for all } t \in\left(0, T_{\max , \varepsilon}\right) \text { and } \varepsilon \in(0,1). 
\end{align*}
This completes the proof of \eqref{4.23-n3.101}.
By a direct integration in \eqref{4.24-2050}, we obtain 
\begin{align*}
\int_0^{T_{\max , \varepsilon}}  \int_{\Omega} v_{\varepsilon}\left|\nabla u_{\varepsilon}\right|^2 \leq C(K)\quad \text { for all } \varepsilon \in(0,1)
\end{align*}
\begin{align*}
\int_0^{T_{\max , \varepsilon}}  \int_{\Omega} u_{\varepsilon} v_{\varepsilon}^{-3}\left|\nabla v_{\varepsilon}\right|^4 \leq C(K)\quad \text { for all } \varepsilon \in(0,1)
\end{align*}
and  
\begin{align*}
\int_0^{T_{\max , \varepsilon}} \int_{\Omega} v_{\varepsilon}^{-1}\left|\nabla v_{\varepsilon}\right|^2\left|D^2 \ln v_{\varepsilon}\right|^2 \leq C(K)\quad \text { for all } \varepsilon \in(0,1),
\end{align*}
which together with \eqref{4.19-3-a} yields
\begin{align*}
\int_0^{T_{\max , \varepsilon}} \int_{\Omega} \frac{\left|\nabla v_{\varepsilon}\right|^6}{v_{\varepsilon}^5} \leq C(K)\quad \text { for all } \varepsilon \in(0,1).
\end{align*}
\end{proof}

We can now ensure that despite the diffusion degeneracy in the first equation of \eqref{-2.5}, the first solution components remain bounded concerning the norm in any $L^p$ space with $p \geq 2$. Our derivation of this will rely on \eqref{4.23-n3.101}, Lemma \ref{4.7-lem-2-1} and the corresponding decay features expressed in Lemma \ref{lemma-4.23-23:32}.
\begin{lem}\label{lem-4.41a}
Let $K>0$. Then for all $p \geq 1$, there exists $C(p, K)>0$ such that if $u_0$ and $v_0$ satisfy \eqref{th1-2} as well as
\begin{align}\label{4.7-5}
\int_{\Omega} u_0^p \leq K
\end{align}
then 
\begin{align}\label{4.7-6}
\int_{\Omega} u_{\varepsilon}^p(t) \leq C(p, K) \quad \text { for all } t \in (0,T_{\max , \varepsilon}) \text { and } \varepsilon \in(0,1)\quad \text { for all } \varepsilon \in(0,1)
\end{align}
and
\begin{align}\label{4.7-611}
\int_0^{T_{\max , \varepsilon}} \int_{\Omega} u_{\varepsilon}^{p-1} v_{\varepsilon} \left|\nabla u_{\varepsilon}\right|^2 \leq  C(p, K) \quad \text { for all } \varepsilon \in(0,1)
\end{align}
and
\begin{align}\label{4.7-6111}
\int_0^{T_{\max , \varepsilon}} \int_{\Omega} u_{\varepsilon}^{p+1} v_{\varepsilon} \leq  C(p, K) \quad \text { for all } \varepsilon \in(0,1)
\end{align}
as well as
\begin{align}\label{4.7-61111}
\int_0^{T_{\max , \varepsilon}}\left\|\nabla\left(u_{\varepsilon}^{\frac{p+1}{2}}(s) v_{\varepsilon}(s)\right)\right\|_{L^2(\Omega)}^2 d s\leq  C(p, K) \quad \text { for all } \varepsilon \in(0,1).
\end{align}
\end{lem}
\begin{proof}
Multiplying the first equation of \eqref{-2.5} by $u_{\varepsilon}^{p-1}$ with $p \geq 1$, integrating by parts and using Young's inequality, we have 
\begin{align}\label{3.7-3}
&\frac{d}{d t} \int_{\Omega} u_{\varepsilon}^p\nonumber\\  
= &p \int_{\Omega} u_{\varepsilon}^{p-1}[\nabla \cdot(u_{\varepsilon} v_{\varepsilon} \nabla u_{\varepsilon})-\chi \nabla \cdot\left(u_{\varepsilon}^{2} v_{\varepsilon} \nabla v_{\varepsilon}\right)+\ell u_{\varepsilon} v_{\varepsilon}] \nonumber\\
= &-p(p-1) \int_{\Omega} u_{\varepsilon}^{p-1} v_{\varepsilon}\left|\nabla u_{\varepsilon}\right|^2+p (p-1)\chi \int_{\Omega} u_{\varepsilon}^{p} v_{\varepsilon} \nabla u_{\varepsilon} \cdot \nabla v_{\varepsilon} + p \ell \int_{\Omega} u_{\varepsilon}^{p} v_{\varepsilon}\nonumber\\
\leq & -\frac{p(p-1)}{2} \int_{\Omega} u_{\varepsilon}^{p-1} v_{\varepsilon} \left|\nabla u_{\varepsilon}\right|^2\nonumber\\
& +\frac{p(p-1) \chi^2}{2} \int_{\Omega} u_{\varepsilon}^{p+1} v_{\varepsilon}\left|\nabla v_{\varepsilon}\right|^2+p \ell \int_{\Omega} u_{\varepsilon}^{p+1}  v _{\varepsilon} + p \ell \int_{\Omega}  v_{\varepsilon}
\end{align}
for all $t \in(0, T_{\max , \varepsilon})$ and $\varepsilon \in(0,1)$. Using Lemma \ref{4.7-lem-2-1} with $\eta =\min \left\{\frac{1}{2 \chi^2}, 1\right\}$, we provide $c_2>0$ such that 
\begin{align*}
\frac{p(p-1) \chi^2}{2} \int_{\Omega} u_{\varepsilon}^{p+1} v_{\varepsilon}\left|\nabla v_{\varepsilon}\right|^2 & \leq   \frac{p(p-1) }{4} \int_\Omega u_{\varepsilon}^{p-1} v_{\varepsilon}|\nabla u_{\varepsilon}|^2+  \int_{\Omega} u_{\varepsilon} v_{\varepsilon}\nonumber\\
&+  c_2 \cdot \left(\left\| v_{\varepsilon} \right\|_{L^{\infty}(\Omega)}+\left\|v_{\varepsilon} \right\|^3_{L^{\infty}(\Omega)}\right) \cdot \int_{\Omega} u_{\varepsilon}^{p+1} v_{\varepsilon}  \cdot \int_\Omega \frac{|\nabla v_{\varepsilon}|^4}{v_{\varepsilon}^3} \nonumber\\
&+c_2 \cdot \left\|v_{\varepsilon} \right\|^4_{L^{\infty}(\Omega)} \cdot\left\{\int_\Omega u_{\varepsilon} \right\}^{2 p+1} \cdot \int_\Omega \frac{|\nabla v_{\varepsilon}|^4}{v_{\varepsilon}^3}.
\end{align*}
for all $t \in(0, T_{\max , \varepsilon})$ and $\varepsilon \in(0,1)$. This together with \eqref{-2.9}, \eqref{4.23-n3.101} and \eqref{3.7-3} yields with $c_3:=c_3(K)>0$
\begin{align}\label{4.24-17:01}
&  \frac{d}{d t} \int_{\Omega} u_{\varepsilon}^p 
+ \frac{p(p-1)}{4} \int_{\Omega} u_{\varepsilon}^{p-1} v_{\varepsilon} \left|\nabla u_{\varepsilon}\right|^2\nonumber\\
\leq &  c_2 \cdot \left(\left\| v_{\varepsilon} \right\|_{L^{\infty}(\Omega)}+\left\|v_{\varepsilon} \right\|^3_{L^{\infty}(\Omega)}\right) \cdot \int_{\Omega} u_{\varepsilon}^{p+1} v_{\varepsilon}  \cdot \int_\Omega \frac{|\nabla v_{\varepsilon}|^4}{v_{\varepsilon}^3}\nonumber\\
& +c_2 \cdot \left\|v_{\varepsilon} \right\|^4_{L^{\infty}(\Omega)} \cdot\left\{\int_\Omega u_{\varepsilon} \right\}^{2 p+1} \cdot \int_\Omega \frac{|\nabla v_{\varepsilon}|^4}{v_{\varepsilon}^3} +\int_{\Omega} u_{\varepsilon} v_{\varepsilon} + p \ell \int_{\Omega} u_{\varepsilon}^{p+1}  v _{\varepsilon} + p \ell \int_{\Omega}  v_{\varepsilon}\nonumber\\
\leq &  c_2 c_3 \cdot \left(\left\| v_{0} \right\|_{L^{\infty}(\Omega)}+\left\|v_{0} \right\|^3_{L^{\infty}(\Omega)}\right) \cdot \int_{\Omega} u_{\varepsilon}^{p+1} v_{\varepsilon}\nonumber\\
& +c_2 \cdot \left\|v_{0} \right\|^4_{L^{\infty}(\Omega)} \cdot\left\{\int_\Omega u_{\varepsilon} \right\}^{2 p+1} \cdot \int_\Omega \frac{|\nabla v_{\varepsilon}|^4}{v_{\varepsilon}^3} +\int_{\Omega} u_{\varepsilon} v_{\varepsilon} + p \ell \int_{\Omega} u_{\varepsilon}^{p+1}  v _{\varepsilon} + p \ell \int_{\Omega}  v_{\varepsilon}
\end{align}
From Lemma \ref{lemma-4.23-23:32}, we have with $c_4:=c_4(p)>0$, 
\begin{align}\label{513-2043}
\int_{\Omega} u_{\varepsilon}^{p +1} v_{\varepsilon} \leq  c_4 \cdot \left\{\int_{\Omega} \frac{u_{\varepsilon}}{v_{\varepsilon}}|\nabla v_{\varepsilon}|^2+\int_{\Omega} \frac{v_{\varepsilon}}{u_{\varepsilon}}|\nabla u_{\varepsilon}|^2+ \int_{\Omega} u_{\varepsilon} v_{\varepsilon}  \right\}\cdot \int_{\Omega} u_{\varepsilon}^p .
\end{align}
for all $t \in(0, T_{\max , \varepsilon})$ and $\varepsilon \in(0,1)$. Combining \eqref{4.19-1} with \eqref{4.24-17:01} and \eqref{513-2043}, we see that with $c_5:= c_2 c_3 \cdot \left\| v_{0} \right\|_{L^{\infty}(\Omega)}+c_2 c_3 \cdot\left\|v_{0} \right\|^3_{L^{\infty}(\Omega)}$ and $c_6:=c_2 \cdot \left\|v_{0} \right\|^4_{L^{\infty}(\Omega)} \cdot c_0^{2 p+1}(K) $,
\begin{align}\label{513-2048}
& \frac{d}{d t} \int_{\Omega} u_{\varepsilon}^p 
+ \frac{p(p-1)}{4} \int_{\Omega} u_{\varepsilon}^{p-1} v_{\varepsilon} \left|\nabla u_{\varepsilon}\right|^2\nonumber\\
\leq &  c_5 \int_{\Omega} u_{\varepsilon}^{p+1} v_{\varepsilon} +c_2 \cdot \left\|v_{0} \right\|^4_{L^{\infty}(\Omega)} \cdot\left\{\int_\Omega u_{\varepsilon} \right\}^{2 p+1} \cdot \int_\Omega \frac{|\nabla v_{\varepsilon}|^4}{v_{\varepsilon}^3} \nonumber\\
&+\int_{\Omega} u_{\varepsilon} v_{\varepsilon} + p \ell \int_{\Omega} u_{\varepsilon}^{p+1}  v _{\varepsilon} + p \ell \int_{\Omega}  v_{\varepsilon}\nonumber\\
\leq & c_4c_5 \cdot \left\{\int_{\Omega} \frac{u_{\varepsilon}}{v_{\varepsilon}}|\nabla v_{\varepsilon}|^2+\int_{\Omega} \frac{v_{\varepsilon}}{u_{\varepsilon}}|\nabla u_{\varepsilon}|^2+ \int_{\Omega} u_{\varepsilon} v_{\varepsilon}  \right\}\cdot \int_{\Omega} u_{\varepsilon}^p \nonumber\\
& +c_6 \int_\Omega \frac{|\nabla v_{\varepsilon}|^4}{v_{\varepsilon}^3} +\int_{\Omega} u_{\varepsilon} v_{\varepsilon} + p \ell \int_{\Omega} u_{\varepsilon}^{p+1}  v _{\varepsilon} + p \ell \int_{\Omega}  v_{\varepsilon}\nonumber\\
\leq &  (c_4c_5+p\ell)\cdot\left\{\int_{\Omega} \frac{u_{\varepsilon}}{v_{\varepsilon}}|\nabla v_{\varepsilon}|^2+\int_{\Omega} \frac{v_{\varepsilon}}{u_{\varepsilon}}|\nabla u_{\varepsilon}|^2+ \int_{\Omega} u_{\varepsilon} v_{\varepsilon}  \right\}\cdot \int_{\Omega} u_{\varepsilon}^p \nonumber\\
& +c_6 \int_\Omega \frac{|\nabla v_{\varepsilon}|^4}{v_{\varepsilon}^3} +\int_{\Omega} u_{\varepsilon} v_{\varepsilon}  + p \ell \int_{\Omega}  v_{\varepsilon} 
\end{align}
for all $t \in(0, T_{\max , \varepsilon})$ and $\varepsilon \in(0,1)$. For each $\varepsilon \in(0,1)$, set
\begin{align*}
y_{\varepsilon}(t) = \int_{\Omega} u_{\varepsilon}^p
\end{align*}
and
\begin{align*}
g_{\varepsilon}(t)=c_6 \int_\Omega \frac{|\nabla v_{\varepsilon}|^4}{v_{\varepsilon}^3} +\int_{\Omega} u_{\varepsilon} v_{\varepsilon} + p \ell \int_{\Omega}  v_{\varepsilon}
\end{align*}
as well as
\begin{align*}
h_{\varepsilon}(t)=(c_4c_5 +p\ell)\cdot \left\{\int_{\Omega} \frac{u_{\varepsilon}}{v_{\varepsilon}}|\nabla v_{\varepsilon}|^2+\int_{\Omega} \frac{v_{\varepsilon}}{u_{\varepsilon}}|\nabla u_{\varepsilon}|^2+ \int_{\Omega} u_{\varepsilon} v_{\varepsilon}  \right\},
\end{align*}
thus \eqref{513-2048} becomes 
\begin{align*}
y_{\varepsilon}^{\prime}(t) \leq g_{\varepsilon}(t) + h_{\varepsilon}(t) y_{\varepsilon}(t),\quad \text { for all } t \in\left(0, T_{\max , \varepsilon}\right) 
\end{align*}
which upon an ODE comparison argument implies that
\begin{align}\label{4.24-16:58}
y_{\varepsilon}(t) \leq y_{\varepsilon}(0) e^{\int_0^t h_{\varepsilon}(s) d s}+\int_0^t e^{\int_s^t h_{\varepsilon}(\sigma) d \sigma} g_{\varepsilon}(s) d s \quad \text { for all } t \in\left(0, T_{\max , \varepsilon}\right) .
\end{align}
Since 
\begin{align*}
\int_s^t h_{\varepsilon}(\sigma) d \sigma \leq c_7(p,K)  \quad \text { for all } t \in\left(0, T_{\max , \varepsilon}\right), s \in[0, t) \text {, and } \varepsilon \in(0,1)
\end{align*}
by \eqref{-2.9}, \eqref{-2.10}, \eqref{n-3.4} and \eqref{n-3.9}, and since  
\begin{align*}
\int_0^t g_{\varepsilon}(s) d s \leq c_8(p,K) \quad \text { for all } t \in\left(0, T_{\max , \varepsilon}\right)\text { and } \varepsilon \in(0,1).
\end{align*}
due to \eqref{-2.9}, \eqref{-2.10}, \eqref{n-3.4}, \eqref{n-3.5}, \eqref{n-3.9} and \eqref{n-3.10}. From \eqref{4.7-5} and \eqref{4.24-16:58}, we obtain
\begin{align*}
\int_{\Omega} u_{\varepsilon}^p(t) \leq & K e^{c_7(p,K) }+ c_8(p,K) e^{c_7(p,K)}
\end{align*}
for all $t \in(0, T_{\max , \varepsilon})$ and $\varepsilon \in(0,1)$. By a direct integration in \eqref{4.24-17:01}, we obtain
\begin{align*}
\frac{p(p-1)}{4} \int_0^{T_{\max , \varepsilon}} \int_{\Omega} u_{\varepsilon}^{p-1} v_{\varepsilon} \left|\nabla u_{\varepsilon}\right|^2 \leq C(p,K)
\end{align*} 
for all $\varepsilon \in(0,1)$. \eqref{4.7-6} together with \eqref{-2.10}, \eqref{n-3.4} \eqref{n-3.9} and \eqref{513-2043} yields \eqref{4.7-6111}.
To prove \eqref{4.7-6111}, we employ Cauchy-Schwarz inequality to see that
\begin{align*}
& \int_{\Omega}\left|\nabla\left(u_{\varepsilon}^{\frac{p+1}{2}} v_{\varepsilon}\right)\right|^2 \\
\leq & \frac{(p+1)^2}{2} \int_{\Omega} u_{\varepsilon}^{p-1} v_{\varepsilon}^2\left|\nabla u_{\varepsilon}\right|^2+2 \int_{\Omega} u_{\varepsilon}^{p+1}\left|\nabla v_{\varepsilon}\right|^2 \\
\leq & \frac{\left\|v_0\right\|_{L^{\infty}(\Omega)}(p+1)^2}{2} \int_{\Omega} u_{\varepsilon}^{p-1} v_{\varepsilon}\left|\nabla u_{\varepsilon}\right|^2+\int_{\Omega} u_{\varepsilon} v_{\varepsilon}^{-3}\left|\nabla v_{\varepsilon}\right|^4+\int_{\Omega} u_{\varepsilon}^{2 p+1} v_{\varepsilon}^3 \\
\leq & \frac{\left\|v_0\right\|_{L^{\infty}(\Omega)}(p+1)^2}{2} \int_{\Omega} u_{\varepsilon}^{p-1} v_{\varepsilon}\left|\nabla u_{\varepsilon}\right|^2+\int_{\Omega} u_{\varepsilon} v_{\varepsilon}^{-3}\left|\nabla v_{\varepsilon}\right|^4+\left\|v_0\right\|_{L^{\infty}(\Omega)}^2 \int_{\Omega} u_{\varepsilon}^{2 p+1} v_{\varepsilon}
\end{align*}
for all $t \in\left(0, T_{\max , \varepsilon}\right)$ and $\varepsilon \in(0,1)$. Upon integration in time, we have
\begin{align*}
\int_0^{T_{\max , \varepsilon}} \int_{\Omega}\left|\nabla\left(u_{\varepsilon}^{\frac{p+1}{2}} v_{\varepsilon}\right)\right|^2 \leq & \frac{\left\|v_0\right\|_{L^{\infty}(\Omega)}(p+1)^2}{2} \int_0^{T_{\max , \varepsilon}}  \int_{\Omega} u_{\varepsilon}^{p-1} v_{\varepsilon}\left|\nabla u_{\varepsilon}\right|^2\\
& +\int_0^{T_{\max , \varepsilon}}  \int_{\Omega} u_{\varepsilon} v_{\varepsilon}^{-3}\left|\nabla v_{\varepsilon}\right|^4+\left\|v_0\right\|_{L^{\infty}(\Omega)}^2 \int_0^{T_{\max , \varepsilon}}  \int_{\Omega} u_{\varepsilon}^{2 p+1} v_{\varepsilon}
\end{align*}
for all $\varepsilon \in(0,1)$. In view of \eqref{4.23-n3.101-q}, \eqref{4.7-611} and   \eqref{4.7-6111}, we have \eqref{4.7-61111}.
\end{proof}

\section{Global existence of weak solutions}\label{section4}

On the basis of standard heat semigroup estimates we can obtain uniform boundedness of the signal gradient.

\begin{lem}\label{lem-4.25-1239}
Let $K>0$. Then there exists $C(K)>0$ if $u_0$ and $v_0$ satisfy \eqref{th1-2},
then 
\begin{align}\label{4.24-2241}
\left\|v_{\varepsilon}(\cdot,t)\right\|_{W^{1, \infty}(\Omega)} \leq C(K) \quad \text { for all } t \in\left(0, T_{\max , \varepsilon}\right)\text { and } \varepsilon \in(0,1).
\end{align}
\end{lem}
\begin{proof}
According to the Neumann heat semigroup \cite{2010-JDE-Winkler}, fixing any $p>2$ we can find $c_1>0$ such that for all $t \in\left(0, T_{\max , \varepsilon}\right)$ and $\varepsilon \in(0,1)$,
\begin{align}\label{4.24-2317}
&\left\|\nabla v_{\varepsilon}(\cdot,t)\right\|_{L^{\infty}(\Omega)} \nonumber\\
= &\left\|\nabla e^{t(\Delta-1)} v_0-\int_0^t \nabla e^{(t-s)(\Delta-1)}\left\{u_{\varepsilon}(\cdot,s) v_{\varepsilon}(\cdot,s)-v_{\varepsilon}(\cdot,s)\right\} d s\right\|_{L^{\infty}(\Omega)} \nonumber\\
\leq &  c_1\left\|v_0\right\|_{W^{1, \infty}(\Omega)}\nonumber\\
& +  c_1 \int_0^t\left(1+(t-s)^{-\frac{1}{2}-\frac{n}{2 p}}\right) e^{-(t-s)}\left\|u_{\varepsilon}(\cdot,s) v_{\varepsilon}(\cdot,s)-v_{\varepsilon}(\cdot,s)\right\|_{L^p(\Omega)} d s .
\end{align}
Since \eqref{-2.9} implies that
\begin{align*}
& \left\|u_{\varepsilon}(\cdot,s) v_{\varepsilon}(\cdot,s)-v_{\varepsilon}(s)\right\|_{L^p(\Omega)} \\
\leq & \left\|u_{\varepsilon}(\cdot,s)\right\|_{L^p(\Omega)}\left\|v_{\varepsilon}(\cdot,s)\right\|_{L^{\infty}(\Omega)}+|\Omega|^{\frac{1}{p}}\left\|v_{\varepsilon}(\cdot,s)\right\|_{L^{\infty}(\Omega)} \\
\leq & c_2\left\|v_0\right\|_{L^{\infty}(\Omega)}+|\Omega|^{\frac{1}{p}}\left\|v_0\right\|_{L^{\infty}(\Omega)} \quad \text { for all } s \in\left(0, T_{\max , \varepsilon}\right) \text { and } \varepsilon \in(0,1),
\end{align*}
with $c_2(K):=\sup _{\varepsilon \in(0,1)} \sup _{t \in\left(0, T_{\max , \varepsilon}\right)}\left\|u_{\varepsilon}(\cdot,s)\right\|_{L^p(\Omega)}$ being finite by Lemma \ref{lem-4.41a}, from \eqref{4.24-2317} we directly obtain \eqref{4.24-2241}.
\end{proof}
The following two lemmas come from \cite{2024-JDE-Winkler}.
\begin{lem}\label{lem-511-1522}
Let $p_{\star}>2$. Then there exist $\kappa=\kappa\left(p_{\star}\right)>0$ and $K=K\left(p_{\star}\right)>0$ such that for any choice of $p \geq p_{\star}$ and $\eta \in(0,1]$, all $\varphi \in C^1(\bar{\Omega})$ and $\psi \in C^1(\bar{\Omega})$ fulfilling $\varphi>0$ and $\psi>0$ in $\bar{\Omega}$ satisfy
\begin{align*}
\int_{\Omega} \varphi^{p+1} \psi \leq & \eta \int_{\Omega} \varphi^{p-1} \psi|\nabla \varphi|^2+\eta \cdot\left\{\int_{\Omega} \varphi^{\frac{p}{2}}\right\}^{\frac{2(p+1)}{p}} \cdot \int_{\Omega} \frac{|\nabla \psi|^6}{\psi^5} \nonumber\\
& +K \eta^{-\kappa} p^{2 \kappa} \cdot\left\{\int_{\Omega} \varphi^{\frac{p}{2}}\right\}^2 \cdot \int_{\Omega} \varphi \psi
\end{align*}
\end{lem}

\begin{lem}\label{lem-511-1524}
Let $a \geq 1, b \geq 1, d \geq 0$ and $\left(M_k\right)_{k \in\{0,1,2,3, \ldots\}} \subset[1, \infty)$ be such that
\begin{align*}
M_k \leq a^k M_{k-1}^{2+d \cdot 2^{-k}}+b^{2^k} \quad \text { for all } k \geq 1 .
\end{align*}
Then
\begin{align*}
\liminf _{k \rightarrow \infty} M_k^{\frac{1}{2^k}} \leq\left(2 \sqrt{2} a^3 b^{1+\frac{d}{2}} M_0\right)^{e^{\frac{d}{2}}} .
\end{align*}
\end{lem}

We can now complete the most crucial step in the derivation of Theorem \ref{th1} by a recursive $L^p$ regularity argument in Moser's iterative style. However, when we try to bound $\int_{\Omega} u_{\varepsilon}^p$ we only partially rely on parabolicity due to the cross degeneracy of the diffusion in \eqref{-2.5}, namely, growth-limiting
zero-order absorptive expressions exclusively containing $\int_{\Omega} u_{\varepsilon}^p$ will obviously not be governed by respective diffusion-induced integrals. Instead, our argument need to make appropriate use of the temporal decay features expressed in \eqref{5.3-13.46} and \eqref{-2.10}.
\begin{lem}\label{lem-4.25-1259}
There exist $C>0$ and $p_0  > 1$ if $u_0$ and $v_0$ satisfy \eqref{th1-2}, then
\begin{align*}
\left\|u_{\varepsilon}(\cdot,t)\right\|_{L^{\infty}(\Omega)} \leq C \quad \text { for all } t \in\left(0, T_{\max , \varepsilon}\right) \text { and any } \varepsilon \in(0,1) \text {. }
\end{align*}
\end{lem}
\begin{proof}
By Lemma \ref{lem-4.25-1239}, we can find $c_1>0$ such that
\begin{align}\label{511-1312}
\left|\nabla v_{\varepsilon}(x, t)\right| \leq c_1 \quad \text { for all } x \in \Omega, t \in\left(0, T_{\max , \varepsilon}\right) \text { and } \varepsilon \in(0,1) \text {, }
\end{align}
For integers $k \geq 1$ we let $p_k:=2^k p_0$, and let
\begin{align}\label{511-1223-1}
N_{k, \varepsilon}(T):=1+\sup _{t \in(0, T)} \int_{\Omega} u_{\varepsilon}^{p_k}(t)< \infty, \quad T \in\left(0, T_{\max , \varepsilon}\right), k \in\{0,1,2, \ldots\}, \varepsilon \in(0,1).
\end{align}
Due to \eqref{4.7-6}, there exists $c_2>0$ such that 
\begin{align}\label{511-1223}
N_{0, \varepsilon}(T) \leq c_2 \quad \text { for all } T \in\left(0, T_{\max , \varepsilon}\right) \text { and } \varepsilon \in(0,1). 
\end{align}
To estimate $N_{k, \varepsilon}(T)$ for $k \geq 1, T \in\left(0, T_{\max , \varepsilon}\right)$ and $\varepsilon \in(0,1)$, fixing $k$ and $\varepsilon$, according to \eqref{-2.5}, \eqref{511-1312} and Young's inequality, we see that 
\begin{align}\label{511-1339}
\frac{d}{d t} \int_{\Omega} u_{\varepsilon}^{p_k}= & -p_k\left(p_k-1\right) \int_{\Omega} u_{\varepsilon}^{p_k-1} v_{\varepsilon}\left|\nabla u_{\varepsilon}\right|^2+p_k\left(p_k-1\right) \int_{\Omega} u_{\varepsilon}^{p_k} v_{\varepsilon} \nabla u_{\varepsilon} \cdot \nabla v_{\varepsilon}\nonumber \\
& +p_k \ell \int_{\Omega} u_{\varepsilon}^{p_k} v_{\varepsilon}\nonumber \\
\leq & -\frac{p_k\left(p_k-1\right)}{2} \int_{\Omega} u_{\varepsilon}^{p_k-1} v_{\varepsilon}\left|\nabla u_{\varepsilon}\right|^2+\frac{p_k\left(p_k-1\right)}{2} \int_{\Omega} u_{\varepsilon}^{p_k+1} v_{\varepsilon}\left|\nabla v_{\varepsilon}\right|^2 \nonumber\\
& +p_k \ell \int_{\Omega} u_{\varepsilon}^{p_k} v_{\varepsilon}\nonumber \\
\leq & -\frac{p_k\left(p_k-1\right)}{2} \int_{\Omega} u_{\varepsilon}^{p_k-1} v_{\varepsilon}\left|\nabla u_{\varepsilon}\right|^2+ p_k\left(p_k-1\right) c^2_1 \int_{\Omega} u_{\varepsilon}^{p_k+1} v_{\varepsilon} \nonumber \\
& +p_k \ell \int_{\Omega}   u_{\varepsilon}^{p_k} v_{\varepsilon}\nonumber \\
\leq & -\frac{p_k\left(p_k-1\right)}{2} \int_{\Omega} u_{\varepsilon}^{p_k-1} v_{\varepsilon}\left|\nabla u_{\varepsilon}\right|^2\nonumber\\
& + \left\{ p_k\left(p_k-1\right) c^2_1 + p_k \ell \right\}\cdot\left\{\int_{\Omega} u_{\varepsilon}^{p_k+1} v_{\varepsilon}+\int_{\Omega} u_{\varepsilon} v_{\varepsilon}\right\}
\end{align}
Since $p_k\left(p_k-1\right) \leq p_k^2$, $p_k \leq p_k^2$ and $\frac{p_k\left(p_k-1\right)}{2} \geq \frac{p_k^2}{4}$, from \eqref{511-1339} we infer that with $c_3= c^2_1+\ell$,
\begin{align}\label{511-1338}
\frac{d}{d t} \int_{\Omega} u_{\varepsilon}^{p_k}+\frac{p_k^2}{4} \int_{\Omega} u_{\varepsilon}^{p_k-1} v_{\varepsilon}\left|\nabla u_{\varepsilon}\right|^2 \leq c_3 p_k^2 \int_{\Omega} u_{\varepsilon}^{p_k+1} v_{\varepsilon}+c_3 p_k^2 \int_{\Omega} u_{\varepsilon} v_{\varepsilon} \quad \text { for all } t \in\left(0, T_{\max , \varepsilon}\right)
\end{align}
Owing to $p_{\star}=2 p_0> 2$, we can by Lemma \ref{lem-511-1522} assert that with $\kappa:=\kappa\left(p_{\star}\right)>0$ and $K:=K\left(p_{\star}\right)>0$,
\begin{align*}
\int_{\Omega} u_{\varepsilon}^{p_k+1} v_{\varepsilon} \leq & \frac{1}{4 c_3} \int_{\Omega} u_{\varepsilon}^{p_k-1} v_{\varepsilon}\left|\nabla u_{\varepsilon}\right|^2+\frac{1}{4 c_3} \cdot\left\{\int_{\Omega} u_{\varepsilon}^{\frac{p_k}{2}}\right\}^{\frac{2\left(p_k+1\right)}{p_k}} \cdot \int_{\Omega} \frac{\left|\nabla v_{\varepsilon}\right|^6}{v_{\varepsilon}^5} \\
& +K \cdot\left(4 c_3\right)^K \cdot p_k^{2 \kappa} \cdot\left\{\int_{\Omega} u_{\varepsilon}^{\frac{p_k}{2}}\right\}^2 \cdot \int_{\Omega} u_{\varepsilon} v_{\varepsilon} \quad \text { for all } t \in\left(0, T_{\max , \varepsilon}\right) .
\end{align*}
According definition of $\left(p_j\right)_{j \geq 0}$ and \eqref{511-1223-1}, for all $T \in\left(0, T_{\max , \varepsilon}\right)$ we can estimate
\begin{align*}
\int_{\Omega} u_{\varepsilon}^{\frac{p_k}{2}}=\int_{\Omega} u_{\varepsilon}^{p_{k-1}} \leq N_{k-1, \varepsilon}(T) \quad \text { for all } t \in(0, T),
\end{align*}
and obtain that
\begin{align*}
c_3 p_k^2 \int_{\Omega} u_{\varepsilon}^{p_k+1} v_{\varepsilon} \leq & \frac{p_k^2}{4} \int_{\Omega} u_{\varepsilon}^{p_k-1} v_{\varepsilon}\left|\nabla u_{\varepsilon}\right|^2+\frac{p_k^2}{4} N_{k-1, \varepsilon}^{\frac{2\left(p_k+1\right)}{p_k}}(T) \int_{\Omega} \frac{\left|\nabla v_{\varepsilon}\right|^6}{v_{\varepsilon}^5} \\
& +4^\kappa c_3^{\kappa+1} K p_k^{2 \kappa+2} N_{k-1, \varepsilon}^2(T) \int_{\Omega} u_{\varepsilon} v_{\varepsilon} \quad \text { for all } t \in(0, T) .
\end{align*}
Since $p_k^2 \leq p_k^{2 \kappa+2}$ and $1 \leq N_{k-1, \varepsilon}^2(T) \leq N_{k-1, \varepsilon}^{\frac{2\left(p_k+1\right)}{p_k}}(T)$, from \eqref{511-1338} we conclude that with $c_4 =\max \left\{\frac{1}{4}, 4^\kappa c_3^{K+1} K+c_3\right\}$,
\begin{align*}
\frac{d}{d t} \int_{\Omega} u_{\varepsilon}^{p_k} \leq c_4 p_k^{2 \kappa+2} N_{k-1, \varepsilon}^{\frac{2\left(p_k+1\right)}{p_k}}(T) g_{\varepsilon}(t) \text { for all } t \in(0, T) \text {, any } T \in\left(0, T_{\max , \varepsilon}\right) \text { and each } \varepsilon \in(0,1) \text {. }
\end{align*}
where $g_{\varepsilon}(t) = \int_{\Omega} \frac{\left|\nabla v_{\varepsilon}(\cdot,t)\right|^6}{v_{\varepsilon}^5(t)}+\int_{\Omega} u_{\varepsilon}(\cdot,t) v_{\varepsilon}(\cdot,t)$.
Based on \eqref{-2.10} and \eqref{5.3-13.46}, we have with $c_5>0$, 
\begin{align*}
\int_0^{T_{\max , \varepsilon}} h_{\varepsilon}(t) d t \leq c_5 \quad \text { for all } \varepsilon \in(0,1).
\end{align*}
An integration of this shows that for all $t \in(0, T)$, $T \in\left(0, T_{\max , \varepsilon}\right)$ and $\varepsilon \in(0,1)$,
\begin{align*}
\int_{\Omega} u_{\varepsilon}^{p_k} & \leq \int_{\Omega}\left(u_0+\varepsilon\right)^{p_k}+c_4 c_5 p_k^{2 \kappa+2} N_{k-1, \varepsilon}^{\frac{2\left(p_k+1\right)}{p_k}}(T) \\
& \leq|\Omega| \cdot\left\|u_0+1\right\|_{L^{\infty}(\Omega)}^{p_k}+c_4 c_5 p_k^{2 \kappa+2} N_{k-1, \varepsilon}^{\frac{2\left(p_k+1\right)}{p_k}}(T)
\end{align*}
In view of \eqref{511-1223}, this yields that
\begin{align*}
N_{k, \varepsilon}(T) & \leq 1+|\Omega| \cdot\left\|u_0+1\right\|_{L^{\infty}(\Omega)}^{p_k}+c_4 c_5 p_k^{2 \kappa+2} N_{k-1, \varepsilon}^{\frac{2\left(p_k+1\right)}{p_k}}(T) \\
& \leq a^{2^k}+b^k N_{k-1, \varepsilon}^{2+d \cdot 2^{-k}}(T) \quad \text { for all } T \in\left(0, T_{\max , \varepsilon}\right) \text { and } \varepsilon \in(0,1)
\end{align*}
where
\begin{align*}
a=\left(1+|\Omega| \cdot\left\|u_0+1\right\|_{L^{\infty}(\Omega)}\right)^{p_0}, \quad b=\left(\max \left\{2 c_4 c_5 p_0, 1\right\}\right)^{2 \kappa+2} \quad \text { and } \quad d=\frac{2}{p_0},
\end{align*}
due to definition of $\left(p_k\right)_{k \geq 0}$,
\begin{align*}
1+|\Omega| \cdot\left\|u_0+1\right\|_{L^{\infty}(\Omega)}^{p_k} \leq\left(1+|\Omega| \cdot\left\|u_0+1\right\|_{L^{\infty}(\Omega)}\right)^{p_k}=\left\{\left(1+|\Omega| \cdot\left\|u_0+1\right\|_{L^{\infty}(\Omega)}\right)^{p_0}\right\}^{2^k}
\end{align*}
and $\frac{2\left(p_k+1\right)}{p_k}=2+\frac{2}{p_k}=2+\frac{2}{p_0} \cdot 2^{-k}$. Since $k \geq 1$ was arbitrary, we employ Lemma \ref{lem-511-1524} and use \eqref{511-1223} to see that
\end{proof}
\begin{align*}
\left\|u_{\varepsilon}(t)\right\|_{L^{\infty}(\Omega)}^{p_0} & =\liminf _{k \rightarrow \infty}\left\{\int_{\Omega} u_{\varepsilon}^{p_k}(t)\right\}^{\frac{1}{2^k}} \\
& \leq \liminf _{k \rightarrow \infty} N_{k, \varepsilon}^{\frac{1}{2^k}}(T) \\
& \leq\left(2 \sqrt{2} b^3 a^{1+\frac{d}{2}} c_2\right)^{e^{\frac{d}{2}}}
\end{align*}
for all $t \in(0, T)$, $T \in\left(0, T_{\max , \varepsilon}\right)$ and each $\varepsilon \in(0,1)$.
Taking $T \nearrow T_{\max , \varepsilon}$, we arrive at \eqref{511-1312}.

The latter especially rules out any blow-up in the approximate problems:
\begin{lem}\label{5.3-2039}
Suppose that $u_0$ and $v_0$ satisfy \eqref{th1-2}. Then $T_{\max , \varepsilon}=+\infty$ for all $\varepsilon \in(0,1)$.
\end{lem}
\begin{proof}
This immediately follows from Lemma \ref{lem-4.25-1259} when combined with \eqref{-2.7}.
\end{proof}

The $L^\infty$ estimate from Lemma \ref{lem-4.25-1259} enables us to control $v_{\varepsilon}$ from below by a comparison argument.
\begin{lem}\label{512-1649}
For all $T>0$ there exists $C(T)>0$ if $u_0$ and $v_0$ satisfy \eqref{th1-2}, then
\begin{align}\label{4.25-1106}
v_{\varepsilon}(x, t) \geq C(T) \quad \text { for all } x \in \Omega, t \in(0, T) \text {, and } \varepsilon \in(0,1).
\end{align}
\end{lem}
\begin{proof}
Since
\begin{align*}
v_{\varepsilon t} \geq \Delta v_{\varepsilon}-c_1 v_{\varepsilon} \quad \text { in } \Omega \times(0, \infty) \quad \text { for all } \varepsilon \in(0,1),
\end{align*}
with $c_1:=\sup _{\varepsilon \in(0,1)} \sup _{t>0}\left\|u_{\varepsilon}(\cdot,t)\right\|_{L^{\infty}(\Omega)}$ being finite by Lemma \ref{lem-4.25-1259}, from a comparison principle we obtain that
\begin{align*}
v_{\varepsilon}(x, t) \geq\left\{\inf _{\Omega} v_0\right\} \cdot e^{-c_1 t} \quad \text { for all } x \in \Omega, t>0 \text { and } \varepsilon \in(0,1),
\end{align*}
so that the claim results from our positivity assumption on $v_0$ in \eqref{th1-2}.
\end{proof}

In order to facilitate a compactness argument based on an Aubin-Lions lemma, we will use  above estimates to obtain certain time regularity of $u_{\varepsilon}$.
\begin{lem}\label{lem-5.3-2014}
Let $K>0$ and $p>1$. Then we can pick $C(K)>0$ such that whenever \eqref{th1-2} holds, one has the following inequalities
\begin{align}\label{5.3-2020}
\int_0^{\infty}\left\|\partial_t\left(u_{\varepsilon}^{\frac{p+1}{2}}(\cdot,t) v_{\varepsilon}(\cdot,t)\right)\right\|_{\left(W^{3,2}(\Omega)\right)^{\star}} d t \leq C(K) \quad \text { for all } \varepsilon \in(0,1)
\end{align}
and
\begin{align}\label{512-1958}
\int_{\Omega} \ln \frac{\left\|v_0\right\|_{L^{\infty}(\Omega)}}{v_{\varepsilon}(\cdot,t)} \leq C(K) \quad \text { for all } t >0 \text { and } \varepsilon \in(0,1) .
\end{align}
as well as 
\begin{align}\label{513-935}
\int_{\Omega}\left(\ln \frac{1}{u_{\varepsilon}(\cdot,t)}\right)_{+}\leq C(K) \quad \text { for all } t >0 \text { and } \varepsilon \in(0,1) .
\end{align}
\end{lem}
\begin{proof}
Invoking Lemmas \ref{lem-4.25-1239}, \ref{5.3-2039} to gain $c_1=c_1(K)>0$ such that
\begin{align*}
\left\|v_{\varepsilon}(\cdot,t)\right\|_{L^{\infty}(\Omega)}+\left\|\nabla v_{\varepsilon}(\cdot,t)\right\|_{L^{\infty}(\Omega)} \leq c_1 \quad \text { for all } t>0 \text { and } \varepsilon \in(0,1) .
\end{align*}
Given $\phi \in W^{3,2}(\Omega)$ satisfying $\|\phi\|_{W^{3,2}(\Omega)} \leq 1$, which implies the existence of $c_2>0$ such that $\|\phi\|_{L^{\infty}(\Omega)}+\|\nabla \phi\|_{L^{\infty}(\Omega)} \leq c_2$ by the Sobolev inequalities, we obtain from \eqref{-2.5} that,
\begin{align*}
\int_{\Omega} \partial_t\left(u_{\varepsilon}^{\frac{p+1}{2}} v_{\varepsilon}\right) \cdot \phi= & -\frac{p+1}{2} \cdot\left\{\frac{p-1}{2} \int_{\Omega} u_{\varepsilon}^{\frac{p-1}{2}} v_{\varepsilon}^2\left|\nabla u_{\varepsilon}\right|^2 \phi+\int_{\Omega} u_{\varepsilon}^{\frac{p+1}{2}} v_{\varepsilon}\left(\nabla u_{\varepsilon} \cdot \nabla v_{\varepsilon}\right) \phi\right. \\
& +\int_{\Omega} u_{\varepsilon}^{\frac{p+1}{2}} v_{\varepsilon}^2\left(\nabla u_{\varepsilon} \cdot \nabla \phi\right)-\frac{p-1}{2} \int_{\Omega} u_{\varepsilon}^{\frac{p+1}{2}} v_{\varepsilon}^2\left(\nabla u_{\varepsilon} \cdot \nabla v_{\varepsilon}\right) \phi \\
& \left.-\int_{\Omega} u_{\varepsilon}^{\frac{p+3}{2}} v_{\varepsilon}\left|\nabla v_{\varepsilon}\right|^2 \phi-\int_{\Omega} u_{\varepsilon}^{\frac{p+3}{2}} v_{\varepsilon}^2\left(\nabla v_{\varepsilon} \cdot \nabla \phi\right)\right\} \\
& -\frac{p+1}{2} \int_{\Omega} u_{\varepsilon}^{\frac{p-1}{2}}\left(\nabla u_{\varepsilon} \cdot \nabla v_{\varepsilon}\right) \phi-\int_{\Omega} u_{\varepsilon}^{\frac{p+1}{2}}\left(\nabla v_{\varepsilon} \cdot \nabla \phi\right) \\
& +\frac{(p+1) \ell}{2} \int_{\Omega} u_{\varepsilon}^{\frac{p+1}{2}} v_{\varepsilon}^2 \phi-\int_{\Omega} u_{\varepsilon}^{\frac{p+3}{2}} v_{\varepsilon} \phi \\
= & -\frac{p+1}{2} \sum_{i=1}^6 I_i(\varepsilon)+\sum_{i=7}^{10} I_i(\varepsilon) \quad \text { for all } t>0 \text { and } \varepsilon \in(0,1) .
\end{align*}
We use Young's inequality to see that for any choice $\varepsilon \in(0,1)$
\begin{align*}
\left|I_1(\varepsilon)\right| & =\frac{p-1}{2}\left|\int_{\Omega} u_{\varepsilon}^{\frac{p-1}{2}} v_{\varepsilon}^2 |\nabla u_{\varepsilon}|^2 \phi \right| \\
& \leq \frac{c_2(p-1)}{2}\left\{\int_{\Omega} u_{\varepsilon}^{p-1} v_{\varepsilon}^2\left|\nabla u_{\varepsilon}\right|^2+\int_{\Omega} v_{\varepsilon}^2\left|\nabla u_{\varepsilon}\right|^2\right\} \\
& \leq \frac{c_1 c_2(p-1)}{2}\left\{\int_{\Omega} u_{\varepsilon}^{p-1} v_{\varepsilon}\left|\nabla u_{\varepsilon}\right|^2+\int_{\Omega} v_{\varepsilon}\left|\nabla u_{\varepsilon}\right|^2\right\} 
\end{align*}
and
\begin{align*}
\left|I_2(\varepsilon)\right|+\left|I_3(\varepsilon)\right| & =\int_{\Omega} u_{\varepsilon}^{\frac{p+1}{2}} v_{\varepsilon}\left(\nabla u_{\varepsilon} \cdot \nabla v_{\varepsilon}\right) \psi+\int_{\Omega} u_{\varepsilon}^{\frac{p+1}{2}} v_{\varepsilon}^2 \nabla u_{\varepsilon} \cdot \nabla \psi\\
& \leq 2 c_2 \int_{\Omega} u_{\varepsilon}^{p-1} v_{\varepsilon}\left|\nabla u_{\varepsilon}\right|^2+2 c_2 c^2_1 \int_{\Omega} u_{\varepsilon}^2 v_{\varepsilon},
\end{align*}
so that it follows from Cauchy-Schwarz inequality and Young's inequality that
\begin{align*}
\left|I_4(\varepsilon)\right| & =\left|-\frac{p-1}{2} \int_{\Omega} u_{\varepsilon}^{\frac{p+1}{2}} v_{\varepsilon}^2\left(\nabla u_{\varepsilon} \cdot \nabla v_{\varepsilon}\right) \phi\right| \\
& \leq \frac{c_2(p-1)}{4}\left\{\int_{\Omega} u_{\varepsilon}^{p-1} v_{\varepsilon}\left|\nabla u_{\varepsilon}\right|^2+c^4_1\int_{\Omega} u_{\varepsilon}^{2} v_{\varepsilon}\right\}
\quad \text { for all } \varepsilon \in(0,1),
\end{align*}
and
\begin{align*}
\left|I_5(\varepsilon)\right|+\left|I_6(\varepsilon)\right| & =\left|-\int_{\Omega} u_{\varepsilon}^{\frac{p+3}{2}} v_{\varepsilon} |\nabla v_{\varepsilon}|^2 \phi \right|+\left|-\int_{\Omega} u_{\varepsilon}^{\frac{p+3}{2}} v^2_{\varepsilon}(\nabla v_{\varepsilon} \cdot \nabla \phi) \right|, \\
& \leq c_2\left\{\int_{\Omega} u_{\varepsilon}^{p+1} v_{\varepsilon}\left|\nabla v_{\varepsilon}\right|^2+2 c^2_1 \int_{\Omega} u^{2}_{\varepsilon} v_{\varepsilon}+c^2_1 \int_{\Omega} u_{\varepsilon}^{p+1} v_{\varepsilon}\right\}  \quad \text { for all } \varepsilon \in(0,1)
\end{align*}
as well as 
\begin{align*}
|I_7(\varepsilon)|+ |I_8(\varepsilon)|& =\left|-\frac{p+1}{2} \int_{\Omega} u_{\varepsilon}^{\frac{1}{2}}\left(\nabla u_{\varepsilon} \cdot \nabla v_{\varepsilon}\right) \psi \right|-\left|\int_{\Omega} u_{\varepsilon}^{\frac{p-1}{2}} \nabla v_{\varepsilon} \cdot \nabla \psi \right| \\
& \leq \frac{c_2(p+1)}{2} \int_{\Omega} u_{\varepsilon}^{\frac{1}{2}}|\nabla u_{\varepsilon}| |\nabla v_{\varepsilon}| + c_2 \int_{\Omega} u_{\varepsilon}^{\frac{p-1}{2}} |\nabla v_{\varepsilon}| \\
& \leq \frac{c_2(p+1)}{2}  \left\{\int_{\Omega} \frac{u_{\varepsilon}}{v_{\varepsilon}}|\nabla v_{\varepsilon}|^2+\int_{\Omega} v_{\varepsilon} |\nabla
u_{\varepsilon}|^2 \right\} + c_2 \left\{\int_{\Omega}u^{p-2}_{\varepsilon} v_{\varepsilon} + \frac{u_{\varepsilon}}{v_{\varepsilon}}|\nabla v_{\varepsilon}|^2  \right\}\\
& \leq \frac{c_2(p+1)}{2}  \left\{\int_{\Omega} \frac{u_{\varepsilon}}{v_{\varepsilon}}|\nabla v_{\varepsilon}|^2+\int_{\Omega} v_{\varepsilon} |\nabla
u_{\varepsilon}|^2 \right\} + c_2 \left\{\int_{\Omega}u^{p+1}_{\varepsilon} v_{\varepsilon} + \frac{u_{\varepsilon}}{v_{\varepsilon}}|\nabla v_{\varepsilon}|^2+ \int_{\Omega} v_{\varepsilon} \right\}.
\end{align*}
Similarly, 
\begin{align*}
|I_7(\varepsilon)|+ |I_8(\varepsilon)|& =\left|\frac{(p+1) \ell}{2} \int_{\Omega} u_{\varepsilon}^{\frac{p+1}{2}} v_{\varepsilon}^2 \psi \right| + \left| -\int_{\Omega} u_{\varepsilon}^{\frac{p+3}{2}} v_{\varepsilon} \psi\right|\\
\leq & \frac{c_2 (p+1) \ell}{2} \int_{\Omega} u_{\varepsilon}^{\frac{p+1}{2}} v_{\varepsilon}^2 + c_2 \int_{\Omega} u_{\varepsilon}^{\frac{p+3}{2}} v_{\varepsilon}\\
\leq & \frac{c_2c_1 (p+1) \ell}{2}  \left\{\int_{\Omega} u_{\varepsilon}^{p+1} v_{\varepsilon} + c^2_1\int_{\Omega}v_{\varepsilon} \right\} + c_2 \left\{ \int_{\Omega} u_{\varepsilon}^{p+1} v_{\varepsilon}+ \int_{\Omega} u^2_{\varepsilon} v_{\varepsilon} \right\}.
\end{align*}
Thus we could find a positive constant $c_3=c_3(K)$ such that
\begin{align*}
& \left\|\partial_t\left(u_{\varepsilon}^{\frac{p+1}{2}} v_{\varepsilon}\right)\right\|_{\left(W^{3,2}(\Omega)\right)^{\star}} \\
= & \sup_{\|\phi\|_{W^{3,2}(\Omega)} \leq 1}\left|\int_{\Omega} \partial_t\left(u_{\varepsilon}^{\frac{p+1}{2}} v_{\varepsilon}\right) \cdot \phi\right| \\
\leq &   c_3 \cdot\left\{\int_{\Omega} u_{\varepsilon}^{p-1} v_{\varepsilon}\left|\nabla u_{\varepsilon}\right|^2+\int_{\Omega} u_{\varepsilon}^{p+1} v_{\varepsilon}+\int_{\Omega} \frac{u_{\varepsilon}}{v_{\varepsilon}}|\nabla v_{\varepsilon}|^2
+ \int_{\Omega} v_{\varepsilon}\left|\nabla v_{\varepsilon}\right|^2+\int_{\Omega} v_{\varepsilon}\left|\nabla u_{\varepsilon}\right|^2+\int_{\Omega} v_{\varepsilon}\right\}
\end{align*}
for all $\varepsilon \in(0,1)$, due to $\int_{\Omega} u_{\varepsilon}^{2} v_{\varepsilon} \leq \int_{\Omega} u_{\varepsilon}^{p+1} v_{\varepsilon}+\int_{\Omega} v_{\varepsilon}$. As a consequence of \eqref{n-3.5}, \eqref{n-3.9}, \eqref{4.23-n3.10}, \eqref{4.7-611} and Lemma \ref{5.3-2039}, we establish \eqref{5.3-2020} upon integration in time. According to the second equation in \eqref{-2.5}, \eqref{-2.9}, \eqref{-2.10} and \eqref{4.25-1106}, we have with $c_4>0$, 
\begin{align*}
\frac{d}{d t} \int_{\Omega} \ln \frac{\left\|v_0\right\|_{L^{\infty}(\Omega)}}{v_{\varepsilon}}=-\frac{d}{d t} \int_{\Omega} \ln v_{\varepsilon}& =-\int_{\Omega} \frac{\left|\nabla v_{\varepsilon}\right|^2}{v_{\varepsilon}^2}+ \int_{\Omega} u_{\varepsilon}\\
& \leq c_1 \int_{\Omega} u_{\varepsilon} v_{\varepsilon}
\quad \text { for all } t>0 \text { and } \varepsilon \in(0,1),
\end{align*}
which upon an integration in time already yields \eqref{512-1958}. To verify \eqref{513-935}, we note that from \eqref{n-3.7} and Lemma \ref{lem-4.25-1259} it follows that with some $c_5>0$ we have
\begin{align*}
\int_{\Omega}\left(\ln \frac{1}{u_{\varepsilon}(t)}\right)_{+} & =-\int_{\Omega} \ln u_{\varepsilon}(t)+\int_{\left\{u_{\varepsilon}(t)>1\right\}} \ln u_{\varepsilon}(t) \\
& \leq-\int_{\Omega} \ln u_{\varepsilon}(t)+|\Omega| \cdot \ln \left\|u_{\varepsilon}\right\|_{L^{\infty}(\Omega \times(0, \infty))} \\
& \leq c_5 \quad \text { for all } t>0 \text { and } \varepsilon \in(0,1).
\end{align*}
\end{proof}

Since the boundedness of $u_{\varepsilon}$ and $v_{\varepsilon}$ asserted in Lemmas  \ref{lem-4.25-1239}, \ref{lem-4.25-1259} and \ref{512-1649}, H\"{o}lder estimates of $u_{\varepsilon}$, $v_{\varepsilon}$ and $\nabla v_{\varepsilon}$ can be derived from standard parabolic regularity theory.
\begin{lem}\label{jia1}
Suppose that $u_0$ and $v_0$ satisfy \eqref{th1-2}. Then for each $T>0$, there exist $\theta_1=\theta(T) \in(0,1)$ and $C_1(T)>0$ such that
\begin{align}\label{eq4.31-xxx}
\left\|u_{\varepsilon}\right\|_{c^{\theta_1, \frac{\theta_1}{2}}(\bar{\Omega} \times[0, T])} \leq C_1(T) \quad \text { for all } \varepsilon \in(0,1)
\end{align}
and
\begin{align}\label{eq4.31-xxxx}
\left\|v_{\varepsilon}\right\|_{C^{\theta_1, \frac{\theta_1}{2}}(\bar{\Omega} \times[0, T])} \leq C_1(T) \quad \text { for all } \varepsilon \in(0,1) .
\end{align}
Moreover, for each $\tau>0$ and all $T>\tau$, there exist $\theta_2=\theta_2(\tau, T) \in(0,1)$ and $C_2(\tau, T)>0$ such that
\begin{align}\label{eq4.31-xxxxx}
\left\|v_{\varepsilon}\right\|_{C^{2+\theta_2, 1+\frac{\theta_2}{2}}(\bar{\Omega} \times [\tau, T])} \leq C_2(\tau, T) \quad \text { for all } \varepsilon \in(0,1) \text {. }
\end{align}
\end{lem}
\begin{proof}
It follows from Lemmas \ref{lem-4.25-1239}, \ref{lem-4.25-1259}, \ref{lem-4.25-1259} and H\"{o}lder regularity \cite{1993-JDE-PorzioVespri} that \eqref{eq4.31-xxx} and \eqref{eq4.31-xxxx} hold. Combining standard Schauder estimates \cite{1968-Ladyzen} with  \eqref{eq4.31-xxx} and a cut-off argument, we can deduce the validity of \eqref{eq4.31-xxxxx}.
\end{proof}

Based on the preparation in the above, we can now use a standard extraction procedure to construct a limit function $(u, v)$ of \eqref{-2.5}, which is proved to be a global weak bounded solution to the system \eqref{-A1} as documented in Theorem \ref{th1}.

\begin{lem}\label{lem-4.18}
Suppose that \eqref{th1-2} holds. Then there exist $\left(\varepsilon_j\right){ }_{j \in \mathbb{N}} \subset(0,1)$ as well as functions $u$ and $v$ which satisfy \eqref{5.4-1429} with $u > 0$ and $v>0$ a.e. in $\Omega \times(0, \infty)$, such that
\begin{flalign}
& u_{\varepsilon} \rightarrow u \quad \text { in } C_{l o c}^0(\bar{\Omega} \times[0, \infty)) \text {, }\label{eq4.34-xx}\\
& v_{\varepsilon} \rightarrow v \quad \text { in } C_{l o c}^0(\bar{\Omega} \times[0, \infty)) \text { and in } C_{l o c}^{2,1}(\bar{\Omega} \times(0, \infty)) ,\label{eq4.36}\\
& \nabla v_{\varepsilon} \stackrel{\star}{\rightharpoonup} \nabla v \quad \text { in } L^{\infty}(\Omega \times(0, \infty)),\label{eq4.34}\\
&  u_{\varepsilon} \rightarrow u \quad \text { in } L_{l o c}^p(\bar{\Omega} \times[0, \infty)) \text { for all } p \in[1, \infty),\label{eq4.33-aa} \\
&  u_{\varepsilon} \rightarrow u \quad \text { a.e. in } \Omega \times(0, \infty),\label{eq4.33} 
\end{flalign}
as $\varepsilon=\varepsilon_j \searrow 0$, and that $(u, v)$ is a global weak solution of (\ref{-A1}) according to Definition \ref{def2.1}, and we have
\begin{align}\label{eq4.381}
\int_0^{\infty} \int_{\Omega} u v \leq \int_{\Omega} v_0
\end{align}
and 
\begin{align}\label{eq4.38}
\int_{\Omega} u_{0} \leq \int_{\Omega} u(t) \leq \int_{\Omega} u_{0}+\ell \int_{\Omega} v_{0} \text {  and  } \|v(t)\|_{L^{\infty}(\Omega)} \leq\left\|v_0\right\|_{L^{\infty}(\Omega)} \quad \text { for all } t>0,
\end{align}
\end{lem}
\begin{proof}
We fix $r \in(1,2)$, $p \in (1,3)$ and employ Young's inequality to estimate
\begin{align}\label{512-1733}
& \int_0^T \int_{\Omega}\left|\nabla u_{\varepsilon}^2\right|^r  \nonumber\\
= & 2^r \int_0^T \int_{\Omega}\left(u_{\varepsilon}^{p-1} v_{\varepsilon}\left|\nabla u_{\varepsilon}\right|^2\right)^{\frac{r}{2}} \cdot u_{\varepsilon}^{\frac{\left(3-p\right) r}{2}} v_{\varepsilon}^{-\frac{r}{2}} \nonumber\\
\leq &  2^r \int_0^T \int_{\Omega} u_{\varepsilon}^{p-1} v_{\varepsilon}\left|\nabla u_{\varepsilon}\right|^2+2^r \int_0^T \int_{\Omega} u_{\varepsilon}^{\frac{\left(3-p\right) r}{2-r}} v_{\varepsilon}^{-\frac{r}{2-r}} \quad \text { for all } T>0 \text { and } \varepsilon \in(0,1),
\end{align}
so that since $\frac{\left(3-p\right) r}{2-r}$ is nonnegative, we use Lemma \ref{lem-4.25-1259} to see that due to \eqref{4.7-611} and Lemma \ref{512-1649},
\begin{align}\label{512-1734}
\left(u_{\varepsilon}^2\right)_{\varepsilon \in(0,1)} \text { is bounded in } L^r\left((0, T) ; W^{1, r}(\Omega)\right) \quad \text { for all } T>0 \text {. }
\end{align}

The existence of $\left(\varepsilon_j\right)_{j \in \mathbb{N}}$ and nonnegative functions $u$ and $v$ with the properties in \eqref{5.4-1429} and \eqref{eq4.34-xx}-\eqref{eq4.34} follows from Lemmas \ref{lem-4.25-1239}, \ref{jia1} by a straightforward extraction procedure.  \eqref{513-935} together with \eqref{th1-2} as well as \eqref{eq4.34-xx} and Fatou’s lemma ensures that $u$ is positive a.e. in $\Omega \times(0, \infty)$. From \eqref{-2.8},
\eqref{-2.9}, \eqref{-2.10} in conjunction with \eqref{eq4.34-xx}, \eqref{eq4.36} and Fatou’s lemma, we have \eqref{eq4.38} and \eqref{eq4.381}.  

In view of \eqref{4.7-6111}, \eqref{4.7-61111} and Lemma \ref{lem-5.3-2014} and Aubin-Lions lemma (see \cite{1977-nhp-Temam1}) as well as $W^{1,2}(\Omega) \hookrightarrow  \hookrightarrow L^2(\Omega)$, there exists a nonnegative function $z \in L_{l o c}^2(\bar{\Omega} \times[0, \infty))$ such that 
\begin{align}\label{512-2014}
u_{\varepsilon}^{\frac{p+1}{2}} v_{\varepsilon} \rightarrow z \quad \text { a.e. in } \Omega \times(0, \infty) \text { and in } L_{l o c}^2(\bar{\Omega} \times[0, \infty))
\end{align}
as $\varepsilon=\varepsilon_j \searrow 0$. By \eqref{-2.9}, \eqref{eq4.36} and Fatou's lemma, we have $v \leq\left\|v_0\right\|_{L^{\infty}(\Omega)}$. \eqref{512-1958} together with Fatou's lemma guarantees that $\ln v$ belongs to $L_{l o c}^1(\bar{\Omega} \times[0, \infty))$ and $v$ is positive a.e. in $\Omega \times(0, \infty)$. Letting $u:=\left(\frac{z}{v}\right)^{\frac{2}{p+1}}$, we obtain an a.e. in $\Omega \times(0, \infty)$ well-defined nonnegative function $u$ for which we have $u_{\varepsilon}=\left(\frac{z_{\varepsilon}}{v_{\varepsilon}}\right)^{\frac{2}{p+1}} \rightarrow u$ a.e. in $\Omega \times(0, \infty)$ as $\varepsilon=\varepsilon_j \searrow 0$ according to \eqref{eq4.36} and \eqref{512-2014}. Since \eqref{4.7-6} asserts that $\left(u_{\varepsilon}\right)_{\varepsilon \in(0,1)}$ is bounded in $L^{\infty}\left((0, \infty) ; L^p(\Omega)\right)$, this limit belong to this space, and moreover satisfy \eqref{eq4.33-aa} as a consequence of the Vitali convergence theorem.
Owing to \eqref{eq4.33-aa}, we can deduce that \eqref{eq4.33} holds.  

The regularity requirements \eqref{-2.1} and \eqref{-2.2} recorded in Definition \ref{def2.1}, \eqref{-2.3} and \eqref{-2.4} become straightforward consequences of \eqref{eq4.34-xx}-\eqref{512-1734}.
\end{proof}

Our main results about the global weak solvability are summarized as follows:
\begin{proof}[\emph{\textbf{Proof of Theorem  \ref{th1}.}}]
Theorem \ref{th1} is a direct consequence of Lemmas \ref{lem-4.25-1239}, \ref{lem-4.25-1259}  and \ref{lem-4.18}.
\end{proof}
%%%%%%%%%%%%%%%%%%%%%%%%%%%%%%%%%%%%

%\hfill$ \Box$
\bibliographystyle{siam}

\bibliography{indrectsingball}

%\end{thebibliography}

\end{document}